\documentclass{amsart}
\usepackage{amsmath,amsfonts,amssymb,amscd,bm}
\usepackage[colorlinks=true]{hyperref}
\usepackage{tikz}
\usepackage{pgflibraryarrows}
\usepackage{pgflibrarysnakes}
\usepackage{pgflibraryarrows}
\usepackage{pgflibrarysnakes}
\hypersetup{urlcolor=blue, citecolor=blue}
\numberwithin{equation}{section} 

\newtheorem{thm}{Theorem}[section]
\newtheorem{coro}[thm]{Corollary}
\newtheorem{lem}[thm]{Lemma}
\newtheorem{proposition}[thm]{Proposition}

\theoremstyle{definition}

\newtheorem{rem}{Remark}[section]

\def\<{\langle}             \def\>{\rangle}

\newcommand{\beeq}{\begin{equation}}\newcommand{\eneq}{\end{equation}}
\newcommand{\al}{\alpha}    \newcommand{\be}{\beta}
\newcommand{\de}{\delta}    \newcommand{\De}{\Delta}
\newcommand{\ep}{\varepsilon}
  
    \newcommand{\la}{\lambda}
    
\newcommand{\om}{\omega}    
    
\newcommand{\R}{\mathbb{R}}

\newcommand{\Sp}{\mathbb{S}}

\newenvironment{prf}{\noindent {\bf Proof.} }{\endprf\par}
\def \endprf{\hfill  {\vrule height6pt width6pt depth0pt}\medskip}
\newcommand{\pt}{\partial_t}\newcommand{\pa}{\partial}

\newcommand{\les}{{\lesssim}}\newcommand{\ges}{{\gtrsim}}
\newcommand{\supp}{\,\mathop{\!\mathrm{supp}}}

\newcommand{\gm}{\mathfrak{g}}

\numberwithin{equation}{section}

\title[Semilinear wave equations on asymptotically Euclidean manifolds]
      {The blow up of solutions to semilinear wave equations 
      on asymptotically Euclidean manifolds}
\author{Mengyun Liu}
\address{Department of Mathematics\\                	
Zhejiang Sci-Tech University\\                Hangzhou 310018, P. R. China}
\email{mengyunliu@zstu.edu.cn}

\author{Chengbo Wang}\address{School of Mathematical Sciences\\                Zhejiang University\\    Hangzhou 310027, P. R. China}\email{wangcbo@zju.edu.cn}
\urladdr{http://www.math.zju.edu.cn/wang}

\thanks{
The authors were supported in part by
 NSFC 11971428. 
 }

\dedicatory{} \commby{}
\begin{document}
\begin{abstract}
In this paper, we investigate the problem of blow up and sharp upper bound  estimates of the lifespan for the solutions to the semilinear wave equations, posed on asymptotically Euclidean manifolds.
Here the metric is assumed to be exponential perturbation of the 
 spherical symmetric, long range asymptotically Euclidean metric.
One of the main ingredients in our proof is
the construction of (unbounded) positive entire solutions for
 $\Delta_{\gm}\phi_\la=\lambda^{2}\phi_\la$, 
with certain estimates which are uniform for small parameter $\lambda\in (0,\la_0)$.
In addition, 
our argument works equally well for semilinear damped wave equations,
when the coefficient of the dissipation term is integrable (without sign condition) and space-independent.
\end{abstract}



\maketitle
\tableofcontents

\section{Introduction}
In this paper, we investigate the problem of blow up and sharp upper bound  estimates of the lifespan for the solutions to the semilinear wave equations, posed on asymptotically Euclidean manifolds.
In addition, 
our argument works equally well for semilinear damped wave equations,
when the coefficient of the dissipation term is integrable (without sign condition) and space-independent.

Let $(\R^{n}, \gm)$ be a asymptotically Euclidean (Riemannian) manifold, with $n\ge 2$. 
By asymptotically Euclidean, we mean that
$(\R^n, \gm)$ is certain perturbation of the Euclidean space $(\R^n, \gm_0)$. More precisely, we assume
$\gm$ can be decomposed as 
\begin{align}
\label{dl1}
\gm=\gm_{1}+\gm_{2}\ ,
\end{align}
where $\gm_{1}$ is a spherical symmetric, long range perturbation of $\gm_0$, and $\gm_{2}$ is an exponential (short range) perturbation.
By definition, there exists polar coordinates $(r,\omega)$ for $(\R^n,\gm_1)$, in which we can write
\beeq
\label{dl3}
\gm_{1}=K^{2}(r)dr^{2}+r^{2}d\omega^{2}\ , 
\eneq
where $d\om^2$ is the standard metric on the unit sphere $\Sp^{n-1}$, and
\beeq
\label{dl2}
|\pa^{m}_{r}(K-1)|\les  \<r\>^{-m-\rho}
,  m=0, 1, 2.
\eneq
for some given constant $\rho >0$
\footnote{It turns out that, for the purpose of blow up results, we only need the following conditions for $K$, viewed as $C^2$ function of $r\ge 0$,
$$
\lim_{r\to \infty}K(r)=1,\ 
\limsup_{r\to \infty} rK'/K<n-1,\ K'\in L^{1}\cap L^{2}, \ K''\in L^{1}, \  
 rK'\in L^{\infty},\ r^{2} K''\in L^{\infty}\ .
$$
}.
Here and in what follows, 
$\langle x\rangle=\sqrt{1+|x|^2}$, and
we use $A\les B$ ($A\gtrsim B$) to stand for $A\leq CB$ ($A\geq CB$) where the constant $C$ may change from line to line. 
Equipped with the coordinates $x=r\omega$, we have
$$\gm=g_{jk}(x)dx^j dx^k\equiv \sum^{n}_{j,k=1}g_{jk}(x)dx^j dx^k\ ,\ \gm_2=g_{2, jk}(x)dx^j dx^k\ ,$$
where we have used the convention that Latin indices $j$, $k$ range from $1$ to $n$ and the Einstein summation convention for repeated upper and lower indices. 
Concerning $\gm_2$, we assume
it is an exponential (short range) perturbation of $\gm_1$, that is, there exists $\al>0$ so that
\beeq
\label{dlfjia}
|\nabla g_{2,jk}|+|g_{2,jk}|\les e^{-\al\int^{r}_{0}K(\tau)d\tau},\ 
|\nabla^2 g_{2,jk}|\les 1\ .
\eneq
By asymptotically Euclidean and Riemannian assumption, it is clear that 
there exists a constant $\delta_{0}\in(0, 1)$ such that 
\beeq
\label{unelp}
\delta_{0}|\xi|^2\le
g^{jk}\xi_{j}\xi_{k}\le
\delta_{0}^{-1} |\xi|^2, \forall \ \xi\in\R^{n},\ 
K\in (\delta_{0}, 1/\delta_{0})
\ .
\eneq

In this paper, we are interested in the  blow up of solutions for the following semilinear  wave equations with small data, posed on aysmtotically Euclidean manifolds \eqref{dl1}-\eqref{unelp},
\beeq
\label{2.1}
\begin{cases}
\pa^{2}_{t}u-\Delta_{\gm}u=|u|^{p}\\
u(0,x)=\ep u_0(x), u_{t}(0,x)=\ep u_1(x)\\
\end{cases}
\eneq
Here,
$\Delta_{\gm} =
g^{-1/2}\pa_j g^{1/2}\gm^{jk} \pa_k$
 is the standard Laplace-Beltrami operator, with
$g=\det(\gm_{jk})$ and $(\gm^{ij})$ being the inverse of $(\gm_{ij})$.
Concerning the initial data, we assume 
 $(u_0, u_1)$ are nontrivial with
\beeq
\label{hs2}
(u_0, u_1)\in C^{\infty}_{0}(\R^{n}),\ u_0, u_1\ge 0,\  \supp(u_0, u_1)\subset\{x\in \R^{n}: r\le R_{0}\}\ , 
\eneq
 for some $R_{0}>0$.

When $\gm=\gm_{0}$,
the problem was initiated by the work of John \cite{John79} for $n=3$ 
where the critical power, for the problem to admit global solutions for any small data,
was determined to be $p_c(3)=1+\sqrt{2}$.
In general, it is known as the Strauss conjecture \cite{Strauss81}, that the critical exponent $p_{c}(n)$ is the positive root of equation:
$$(n-1)p^{2}-(n+1)p-2=0
(\Leftrightarrow
(n-1)p(p-1)=2(p+1)
)
\ .$$
This conjecture has been essentially proved,
see Georgiev-Lindblad-Sogge
\cite{GLS97},
Sideris \cite{Sideris84} and references therein.
The critical case is also known to blow up in general,
see Schaeffer \cite{Scha85} ($n=2, 3$), Yordanov-Zhang \cite{YorZh06} and Zhou \cite{Zh07}
($n\geq 4$).
   In addition, when there is no global solution, it is also important to obtain the estimates of life span for the solutions, in terms of $\ep$. When $\max(1, 2/(n-1))<p<p_{c}(n)$, $n\geq 2$, it have been proved that
\beeq
c\ep^{\frac{2p(p-1)}{(n-1)p^{2} - (n+1)p - 2}}\leq T_{\ep}\leq C\ep^{\frac{2p(p-1)}{(n-1)p^{2} - (n+1)p - 2}},\eneq
for nontrivial data satisfying \eqref{hs2},
where $T_\ep$ denotes the lifespan and $\ep$ is the size of the initial data, $c$, $C$ are some positive constants, 
see 
Lai-Zhou \cite{LaiZhou14},
 Takamura 
\cite{Taka15} and references therein for history. In addition, when $u_1\neq 0$, $n=2$ and $1<p\leq 2$, we have another estimate:
$$T_{\ep}\sim \left\{
\begin{array}{ll }
\ep^{\frac{p-1}{p-3}} ,     &   1<p<2, \\
\ep^{-1}   (\ln(\ep^{-1}))^{-1/2} ,  &   p=2,
\end{array}\right.$$
see Lindblad \cite{L1990} for $p=2$ and Takamura \cite{Taka15}, Imai-Kato-Takamura-Wakasa \cite{IKTW} for $1<p<2$. In the critical case, $p=p_{c}(n)$, it was conjectured that the lifespan is 
\beeq
\label{724}
\exp(c\ep^{-p(p-1)})\leq T_\ep\le \exp(C\ep^{-p(p-1)})\ .
\eneq
 The lower dimensional cases $n=2, 3$ were proved by Zhou
 \cite{MR1177534,MR1233659}. Takamura-Wakasa \cite{TakaWakasa11} obtained the upper bound when $n\geq 4$, see also Zhou-Han \cite{ZhouHan14criti}, while the lower bound was obtained by Lindblad-Sogge \cite{LdSo96} when $n\leq 8$ or $n\geq 2$ for spherically symmetric initial data. 
 See, e.g., Wang \cite{Wang18} for a complete history. 
  
  In the past 10 years or so, there have been many works concerning the analogs of the problem for general manifolds, including non-trapping asymptotically Euclidean manifolds, black hole space-time and exterior domain.
The problem on exterior domain has been relatively well-understood, where existence with $p>p_c(n)$ for spatial dimension up to four and blow up with $p\le p_c(n)$ have been obtained, see Smith-Sogge-Wang \cite{SSW12},  Lai-Zhou \cite{LaiZhou18},  Sobajima-Wakasa \cite{MoWa18}, and references therein.

The existence theory for asymptotically flat space-times has been well-developed for spatial dimension three and four,
see Sogge-Wang \cite{SW10}, Wang-Yu \cite{WaYu11},
Lindblad-Metcalfe-Sogge-Tohaneanu-Wang \cite{LMSTW},
Metcalfe-Wang \cite{MW17} and Wang \cite{W17}, with help of the local energy estimates and weighted Strichartz estimates. Notice here that, comparing \eqref{dlfjia},
the assumption on $\gm_2$, for non-trapping asymptotically Euclidean manifolds,
 is that
\beeq
\label{eq-g2-ae2}
\nabla^\be g_{2,jk}=\mathcal{O}(\<r\>^{-\rho-1-|\be|}), |\be|\le 3\ .
\eneq

In contrast, much less is known for the blow up theory.
Comparing with the existence theory, it is very natural to expect blow up phenomena for $p\le p_c(n)$, for asymptotically Euclidean manifolds $(\R^n, \gm)$ with $\gm_1$ satisfying 
\eqref{dl3}-\eqref{dl2} and $\gm_2$ 
 satisfying 
\eqref{eq-g2-ae2}, as well as the 
Schwarzschild/Kerr black hole spacetimes. 
For asymptotically Euclidean manifolds
with $\gm_{1}=\gm_0$,
and $\gm_2$ 
 satisfying the stronger exponential assumption
\eqref{dlfjia},
 Wakasa-Yordanov
 \cite{WaYo18-1pub} proved blow up results and obtained the expected sharp upper bound of the lifespan \eqref{724} in the critical case $p=p_c(n)$, which agrees with the lower bound for 
 $n=4$ obtained by the second author
 \cite{W17}.
 For the Schwarzschild black hole spacetime, Lin-Lai-Ming \cite{LinLaiMing19} obtained blow up result for $1<p\le 2$, while Catania-Georgiev \cite{CaGe06} obtained a weaker blow up result for $1<p< p_c(3)$. 

Our first main theorem addresses the blow up problem for asymptotically Euclidean manifolds $(\R^n, \gm)$ with $\gm_1$ satisfying 
\eqref{dl3}-\eqref{dl2} and $\gm_2$ 
 satisfying the stronger
\eqref{dlfjia}.

\begin{thm}
\label{thm-1-AE}
Let $n\geq 2$ and $1<p\le p_{c}(n)$. Consider \eqref{2.1} posed on asymptotically Euclidean manifolds \eqref{dl1}-\eqref{dlfjia}.
 Assuming that
the data are nontrivial satisfying \eqref{hs2},
  then there is $\ep_0>0$ so that for any $\ep\in (0, \ep_0)$, 
  there exist $T\ge 1$ and a unique weak solution $u\in C \dot H^{1/2}_{\rm comp}\cap L^{\frac{2(n+1)}{n-1}}_{t,x}([0,T]\times \R^n)$. Moreover, let $T_\ep$ be the lifespan of the local solution, i.e., $T_\ep:=\sup {T}$, then we have
  $T_\ep<\infty$ and more precisely, there exists a positive constant $C_{0}$ depending only on $n, p, R_{0},  u_0, u_1$, such that 
  \beeq\label{eq-life}
T_\ep\le \left\{
\begin{array}{ll }
  C_0 \ep^{\frac{2p(p-1)}{(n-1)p^{2} - (n+1)p - 2}}    & 1<p<p_c(n);   \\
\exp(C_0\ep^{-p(p-1)})      &   p=p_c(n).
\end{array}\right.
\eneq
In addition, when $n=2$, $1<p<2$ and
$u_1$ does not vanish identically,  the upper bound of the lifespan can be improved to 
\beeq
\label{eq-life-special}
T_\ep\le C_0 \ep^{-\frac{p-1}{3-p}}\ .
 \eneq
Here $\dot{H}^{s}$ with $|s|<n/2$ denotes the standard homogeneous Sobolev space, while $\dot{H}^{s}_{\rm comp}$ denotes $\dot H^s$ with compact support.
\end{thm}
\begin{rem}
The upper bound is sharp in general.
Actually,  in the case of nontrapping
asymptotically Euclidean manifolds \eqref{dl1}-\eqref{unelp},
it is known from the second author \cite{W17} that we have 
  \beeq\label{eq-life-compare}
T_\ep\ge \left\{
\begin{array}{ll }
  c_0 \ep^{\frac{2p(p-1)}{(n-1)p^{2} - (n+1)p - 2}}    & 2\le p<p_c(n), n=3,  \\
\exp(c_0\ep^{-p(p-1)})      &   p=p_c(n), n=4,
\end{array}\right.
\eneq
for some $c_0>0$.
\end{rem}
\begin{rem}
As we have said, we expect similar results hold for  asymptotically Euclidean manifolds  with
\eqref{eq-g2-ae2} instead of
\eqref{dlfjia}. It will be interesting to investigate the blow up theory, as well as the high dimensional existence theory, in this setting.
\end{rem}

For the strategy of proof, we basically follow the test function method of
 Yordanov-Zhang \cite{YorZh06}, Zhou \cite{Zh07} and Wakasa-Yordanov \cite{WaYo18-1pub}. The main innovation in our proof of Theorem \ref{thm-1-AE} is 
 the existence of  a class of  generalized  ``eigenfunctions" for 
 the Laplace-Beltrami operator,
 $\Delta_{\gm}\phi_\la=\lambda^{2}\phi_\la$, with small parameter $\lambda\in (0,\la_0)$ and desired (uniform) asymptotical behavior, see Lemma \ref{elp}.

In recent years, the closely related semilinear damped wave equations
\beeq
\label{2.1'}
\begin{cases}
\pa^{2}_{t}u-\Delta_{\gm}u+b(t)u_{t}=|u|^{p}\\
u(0,x)=\ep u_0(x), u_{t}(0,x)=\ep u_1(x)\\
\end{cases}
\eneq
 have also received much attention. In particular, the problem with $\gm=\gm_0$ and typical damping term $b(t)=\mu (1+t)^{-\be}$, $0<\mu\in\R$, $\be\in \R$ has been extensively investigated. In general, the behavior of solutions of \eqref{2.1'} depends on $\mu$ and $\be$.
For the case $\be <1$, the damping term is strong enough to make the  problem behaves totally different from the wave equations, and the problem has been well-understood. 
For the scale-invariant case $\be=1$, it appears that the critical power is $p_c(n+\mu)$ for relatively small $\mu>0$,
see, e.g.,
 Ikeda-Sobajima \cite{IkSo18},
 Tu-Lin \cite{TuLin17p1}.
See \cite{LaiTa18} for more discussion on the history.

For the remaining case, $\be>1$ (which is also referred as the scattering case), where the damping term is integrable, it is natural to expect that the problem behaves like the nonlinear wave equations without damping term, regardless of the sign of $b(t)$. 
For the blow up part, Lai-Takamura \cite{LaiTa18}  proved blow up results 
with 
$\gm = \gm_{0}$, $0\le b(t)\in L^{1}$ for $1<p<p_c(n)$, together with upper bound of the lifespan
\beeq\label{eq-life-subcrit}
T_\ep\le C \ep^{\frac{2p(p-1)}{(n-1)p^{2} - (n+1)p - 2}}, n\geq 2\ ,
\eneq
under the assumption that
\beeq\label{eq-fsp}\supp u\subset \{(x,t)\in\R^n\times [0, T_\ep): |x|\le t+R\}\eneq
for some $R>0$.
Here, the sign assumption on $b$ is removed in a recent work of Ikeda-Tu-Wakasa \cite{MaTuWa}.
For the critical case, $p=p_c(n)$,
with $\gm_{1}=\gm_0$
 and $0\le b(t)\in L^1$, 
 under the assumption that
\eqref{eq-fsp},
 Wakasa-Yordanov 
\cite{WaYo18-2-pub}
obtained the expected exponential upper bound of the lifespan
\beeq\label{eq-life-crit}
T_\ep\le \exp(C\ep^{-p(p-1)})\ .
\eneq
On the other hand, the existence theory depends on the spatial dimension and the high dimensional results remain open. The global existence for $p>p_c(n)$ and general lower order terms has been verified by the authors \cite{LW2018}
for two dimensional Euclidean space, and 
nontrapping asymptotically Euclidean manifolds $(\R^n, \gm)$ with $n=3, 4$, $\gm_1$ satisfying 
\eqref{dl3}-\eqref{dl2} and $\gm_2$ 
 satisfying 
\eqref{eq-g2-ae2}.

Now, we are ready to present our second main theorem, 
for semilinear damped wave equations
\eqref{2.1'} on asymptotically Euclidean manifolds \eqref{dl1}-\eqref{unelp},
where we merely assume $b(t)\in L^1$
 and there are no any sign condition even for the critical case.
\begin{thm}
\label{thm-2-AE-damp}
Let $n\geq 2$, $b(t)\in L^1(\R_+)$ and $1<p\le p_{c}(n)$. Consider \eqref{2.1'} posed on asymptotically Euclidean manifolds \eqref{dl1}-\eqref{dlfjia}.
 Assuming that
the data are nontrivial satisfying \eqref{hs2}, 
  then we have the same conclusion as in Theorem \ref{thm-1-AE}.
 \end{thm}

For the strategy of proof of Theorem \ref{thm-2-AE-damp}, we perform
 a change of variable
to transfer
the problem \eqref{2.1'} to the equivalent 
problem
\beeq
m^2(t)\pa^{2}_{t}u-\Delta_{\gm}u=|u|^{p}\ , 
\eneq
where $m\simeq 1$,
as long as $b(t)$ is integrable. This is one of the main innovations 
 in our proof of Theorem \ref{thm-2-AE-damp}, which completely avoids the sign condition on $b(t)$. 

\begin{rem}
The proof of local existence of \eqref{2.1} and
\eqref{2.1'}
ensures the obtained solution satisfies finite speed of propagation. Actually,
for linear wave equation $\pa^{2}_{t}u-\Delta_{\gm_1}u+b(t)u_{t}=0$ with the initial data \eqref{hs2}, the support of solution $u$ satisfies
\beeq
\label{fspyl}
\supp u\subset \{(t, x); \int^{|x|}_{0}K(\tau)d\tau\leq t+R_{1}\},
\eneq
where $R_{1}=\int_{0}^{R_{0}}K(\tau)d\tau$.
As $\gm_2$ is short-range perturbation, which does not affect the speed of propagation too much,
we still have \eqref{fspyl}, with possibly bigger $R_1$, for solutions to
linear wave equations $\pa^{2}_{t}u-\Delta_{\gm}u+b(t)u_{t}=0$, and so is the support of weak solutions to
 \eqref{2.1}  and
\eqref{2.1'}. 
This helps removing
the finite speed of propagation assumption 
\eqref{eq-fsp} in previous works.
See Section \ref{sec-lwp} for the proof of local existence and uniqueness of weak solutions.
\end{rem}

 \medskip

\subsubsection*{Outline} Our paper is organized as follows. In the next section, we
present the proof of 
local existence and uniqueness of weak solutions. As corollary, we 
derive the finite speed of propagation, \eqref{fspyl}, as the property of the solutions instead of the assumption.
In \S\ref{sec-elliptic}, we study
the existence of 
certain generalized ``eigenfunction" for elliptic equation $\Delta_{\gm}\phi=\lambda^{2}\phi$, with $0<\la<\la_0$ for sufficiently small $\la_0>0$, which plays a key role in the construction of the test function.
In \S\ref{sec-sub}, we 
present a unified proof of 
Theorems \ref{thm-1-AE}-\ref{thm-2-AE-damp}, in the subcritical case,
by introducing a change of variable.
The last \S\ref{sec-crit} is devoted to the proof of the critical case, where the whole class of eigenfunctions with parameter $0<\la<\la_0$ are exploited.

\section{Local well posedness}\label{sec-lwp}
In this section, we
prove
local existence and uniqueness of weak solutions.

Our proof will rely on the local Strichartz estimates
 for the wave operator $\pa^{2}_{t}-\Delta_{\gm}+b(t)\pa_{t}$.
The local Strichartz estimates for the  wave equations with
variable coefficients
have been  extensively studied, see
Tataru \cite{Ta02} and references therein.
In the setting of asymptotically Euclidean manifolds, 
the following local 
 Strichartz estimates has been obtain in Sogge-Wang \cite[Proposition 5.4]{SW10},
based on Tataru \cite{Ta02}.
\begin{proposition}[Local homogeneous Strichartz estimates]
Let  $n\ge 2$, $b\in L^1$,
$(\R^n,\gm)$ be
asymptotically Euclidean manifold with
$$
\nabla^\be (g_{jk}-\de_{jk})=\mathcal{O}(\<r\>^{-\rho-|\be|}), |\be|\le 2\ ,$$
for some $\rho>0$. 
Then 
for any $s\in [0, 1]$,
and $(s, q, r)$ admissible with $r<\infty$, that is 
$$\frac{1}{q}\leq \min\left(\frac{1}{2}, \frac{n-1}{2}
\left(\frac{1}{2}-\frac{1}{r}\right)\right), s=n\left(\frac{1}{2}-\frac{1}{r}\right)-\frac{1}{q}\ ,$$
we have the local homogeneous Strichartz estimates 
\beeq
\|u\|_{L^{q}_{t\in[0, 1]}L_{x}^{r} \cap L^\infty_{t\in [0,1]} \dot H^s}
+\|\pa u\|_{L^\infty_{t\in [0,1]} \dot H^{s-1}}
\les \|u(0)\|_{\dot{H}^{s}}+\|u_t(0)\|_{\dot{H}^{s-1}}\ ,
\eneq
for solutions to 
$(\pt^2-\Delta_\gm+b(t)\pt)u=0$.
\end{proposition}
Notice that the result was originally proved for $b(t)=0$. The result when $b(t)\in L^{1}$ follows directly if we combine it with Duhamel's principle and the Gronwall's inequality. Moreover, by Duhamel's principle and Christ-Kiselev lemma \cite{ChristK01}, we immediately get the following:
\begin{coro}
Let $s\in [0, 1]$, $r,\tilde{r}<\infty$ and $(s, q, r)$, $(1-s, \tilde{q}, \tilde{r})$ be admissible.
The solution of equation $(\pt^2-\Delta_\gm+b(t)\pt)u=F$ satisfies
\beeq
\label{NoHStri}
\|u\|_{L^\infty_{t\in [0,1]} \dot H^{s}}+
\|u\|_{L^{q}_{t\in[0, 1]}L_{x}^{r}}\les\|u(0)\|_{\dot{H}^{s}}+\|u_t(0)\|_{\dot{H}^{s-1}}+\|F\|_{L^{\tilde{q}'}_{t}L_{x}^{\tilde{r}'}}\ .
\eneq
In particular,
when $q=r=\tilde q=\tilde r=2\frac{n+1}{n-1}$ and $s=1/2$, we have
\beeq
\label{eq-Stri}
\|u\|_{L^\infty_{t} \dot H^{1/2} \cap L^{2\frac{n+1}{n-1}}( [0, 1]\times \R^n)}\les \|u(0)\|_{\dot{H}^{1/2}}+\|u_{t}(0)\|_{\dot{H}^{-1/2}}+
\|F\|_{L^{2\frac{n+1}{n+3}}( [0, 1]\times \R^n)}
\ .
\eneq

\end{coro}
Thus we can apply \eqref{eq-Stri} to \eqref{2.1'} to get the local well posedness.
\begin{lem}\label{thm-lwp}
Let $n\geq 2$ and $1<p\leq\frac{n+3}{n-1}$. Then the Cauchy problem \eqref{2.1'} with initial data \eqref{hs2} is locally well posed in $\dot H^{1/2}_{\rm comp}\times \dot H^{-1/2}_{\rm comp}$ for sufficiently small $\ep$.
\end{lem}
\begin{proof}
With help of the Strichartz estimates \eqref{eq-Stri}, it is easy to prove local well posedness for compactly supported small data, by standard contraction mapping principle. We only give the sketch of the proof of local existence in the time interval $[0,1]$. 
Let $q=2\frac{n+1}{n-1}$ and
$$\|u\|_{X}=\|u\|_{L^\infty_{t\in [0,1]} \dot H^{1/2}}+
\|u\|_{L^{q}( [0, 1]\times \R^n)}.$$
We set $u^{(0)}=0$ and recursively define $u^{(k+1)}$ to be the solution to the linear equation
\beeq
\left\{\begin{array}{l}\pa^{2}_{t}u^{(k+1)} - \Delta_{\gm}u^{(k+1)}+b(t) \pt u^{(k+1)}=|u^{(k)}|^{p},\quad
\\
u^{(k+1)}(0,x) = \ep u_0,
 \pt u^{(k+1)}(0,x) = \ep u_1.\end{array}
 \right.
 \label{11.1}
\eneq
Assuming, by induction, for some $k\ge 0$, $u^{(k)}$ is well-defined in $X$ 
with
$\|u^{(k)}\|_X\les\ep$
and 
 finite speed of propagation
\beeq
\label{fspxz}
\supp u^{(k)}\subset \{(t, x); t\in [0, 1], \int^{|x|}_{0}K(\tau)d\tau\leq t+R_{1}\}:=A
\ .
\eneq
Then we are reduced to prove the boundedness of $u^{(k+1)}$,
as well as convergence, in $X$ 
for small enough $\ep$. We only show the boundedness since the convergence follows the same way.

As $q=2\frac{n+1}{n-1}$, we
observe that
$q/q'=(n+3)/(n-1)$.
Since $1<p\le (n+3)/(n-1)$, we have
$p q'\le q$.
Applying \eqref{eq-Stri} to \eqref{11.1}, 
we have 
\beeq\label{eq-cont}\|u^{(k+1)}\|_{X}\les \ep+\||u^{(k)}|^p\|_{L^{q'}_{t,x}}
\les \ep+\|u^{(k)}\|_{L^{pq'}_{t,x}}^p
\les \ep+\|u^{(k)}\|_{L^{q}_{t,x}}^p,\eneq
where we have used H\"older's inequality in the last inequality and the fact that we have $(t,x)\in A$ which is compact.
Thus the (uniform) boundedness could be easily obtained by continuity and induction from \eqref{eq-cont},
when $\ep$ is smaller than some $\ep_{1}>0$.
\end{proof}

\section{
Positive entire solutions to
the equation $\Delta_{\gm}\phi=\lambda^{2}\phi$}
\label{sec-elliptic}

In this section, for asymptotically Euclidean manifolds \eqref{dl1}-\eqref{unelp}, we prove the existence
 of certain generalized ``eigenfunction", denoted by $\phi_\la(x)$,  for elliptic equation
 \beeq
\label{1.1}
\Delta_{\gm}\phi_\la=\lambda^{2}\phi_\la\ ,
\eneq
 with desired (uniform) asymptotic behavior for any $0<\la\le \la_0$ with sufficiently small $\la_0>0$. 
 Such solutions are also known as positive entire solutions in literature.
 These solutions  will play a key role in the construction of the test functions.

\begin{lem}
\label{elp}
Let $n\geq 2$ and 
$(\R^n, \gm)$ be asymptotically Euclidean manifold with \eqref{dl1}-\eqref{dlfjia}.
Then there exist $\la_0, c_1>0$ such that for any
$0<\la\le \la_0$, there is a solution of \eqref{1.1} satisfying
\beeq
\label{1.50}
c_{1} 
<\phi_\la(x) < c_1^{-1}\< \la r\>^{-\frac{n-1}{2}}e^{\la \int^{r}_{0}K(\tau)d\tau} \ .
\eneq
\end{lem}

\begin{rem}
We could also prove 
the exponential lower bound, \beeq\label{eq-lowerbd}\phi_\la(x)\gtrsim \langle r\lambda \rangle^{-\frac{n-1}{2}}e^{\lambda\int^{r}_{0}K(\tau)d\tau}\ .\eneq
However, as it turns out that it is not necessary for the proof of the main theorems, we choose to put it in the remark to emphasize this fact.
\end{rem}

To give the proof of Lemma \ref{elp}, we will first present a proof for the case $\gm=\gm_1$, where we could exploit
the spherical symmetric property to
work in the corresponding 
 ordinary differential equation, near spatial infinity. Then due to the exponential feature of the short range perturbation $\gm_2$, a perturbation argument completes the proof of  Lemma \ref{elp}.

\subsection{Asymptotic behavior for certain ordinary differential equation}
As preparation, we present a key lemma which gives us the asymptotic behavior for solutions to certain ordinary differential equation, depending on parameter $\la\in (0,\la_0]$.

\begin{lem}\label{thm-ode}
Let $\la\in (0, \la_0]$, $\de_0\in (0, 1)$, $\ep>0$, $y_0>0$,  $K\in (\de_0, \de_0^{-1})$,
\beeq
\label{0-eq-Gest}
\|K'\|_{L^1([\ep \la_0^{-1},\infty))}\le \de_0^{-1},
\|G\|_{L^1([\ep \la^{-1},\infty))}\le \de_0^{-1} \la, \forall \la\in (0, \la_0]\ .
\eneq
Considering
\beeq
\label{0-1.2}
\begin{cases}
y''-\lambda^{2}K^{2}(r)y+G(r)y=0, r>\ep\la^{-1}\\
y(\ep \la^{-1})=y_0, y'(\ep \la^{-1})=y_{1}\in (0, \de_0^{-1}\la y_0 ),
\end{cases}
\eneq
Then for any solution $y$ 
with $y, y'>0$,
we have the following uniform estimates, independent of $\la\in (0,\la_0]$,
\beeq
\label{0.1-eq-1.6}y
\simeq  y_0 e^{\lambda\int^{r}_{\ep/\la }K(\tau)d\tau}
\ ,\ r\ge \ep\la^{-1}\ .\eneq
Assume in addition
$1-\lambda^{-2}K^{-2}G\in (\de_0, \de_0^{-1})$, 
then 
the solution $y$ to
\eqref{0-1.2} satisfies
 $y, y'>0$ and
we have 
\beeq
\label{0.1-eq-1.6'} y'\simeq y_1+y_0\la (e^{\lambda\int^{r}_{\ep/\la }K(\tau)d\tau}-1)\ .\eneq
\end{lem}
\begin{proof}
By simple change of variable and linearity, we need only to prove for the case $\ep=y_0=1$.
The proof of \eqref{0.1-eq-1.6'} follows directly from \eqref{0.1-eq-1.6},
the equation $y''\sim \la^2 y$, and Newton-Leibniz formula.
In the following, we 
prove  
\eqref{0.1-eq-1.6} for the case $\ep=y_0=1$.

Let  $\mu(t)$ be such that
 \beeq
 \label{0-eq-1.1}
y=e^{\lambda\int^{r}_{1/\lambda}K(\tau)d\tau+\int_{1/\lambda}^{r}\mu(\tau)d\tau}\ .
 \eneq
Then by \eqref{0-1.2}, we have $\mu$ satisfies
\beeq
\label{0-eq-tuo}
\mu'+2\lambda K\mu+\mu^{2}=-G(r)-\lambda K'\ ,
\eneq
and
\beeq
\label{0-eq-1.3}
\mu(r)=\frac{y'}{y}-\lambda K(r)>-\lambda K(r)
>-\de_0^{-1}\la
\ . 
\eneq

By the assumptions on $y_{0}, y_{1}$ and $K>0$,
we have \beeq
\label{0-eq-1.4}
\mu(1/\lambda)= \frac{y_{1}}{y_{0}} -\lambda K(1/\lambda)<
\de_0^{-1}\lambda\ .
\eneq
Moreover, by \eqref{0-eq-tuo}, we have 
$$\mu'+2\lambda K\mu\le -G(r)-\lambda K'\ ,$$
and thus, in view of $K>0$, \eqref{0-eq-Gest} and \eqref{0-eq-1.4}, we have
\beeq
\mu(r) \leq \mu(1/\lambda)
e^{-2\lambda\int^{r}_{1/\lambda}K(\tau)d\tau}
+\int^{r}_{1/\lambda}(|G|+\la|K'|)d\tau
< 3
\de_0^{-1} \lambda
 \ .\label{0-eq-1.5}
\eneq
and so
$|\mu|< 3\de_{0}^{-1}\la$.

On the one hand, if we divide \eqref{0-eq-tuo} by $K$ and integrate from $1/\lambda$ to $r$, by  \eqref{0-eq-Gest}, \eqref{0-eq-1.3}-\eqref{0-eq-1.5} and recall that $K\in (\delta_{0}, 1/\delta_{0})$, we get
\begin{eqnarray*}
2\la \int^{r}_{1/\lambda}\mu d\tau
&\leq&-\int_{1/\lambda}^{r}\frac{G(\tau)}{ K}d\tau-\lambda\int_{1/\lambda}^{r}\frac{K'}{K}d\tau-\int_{1/\lambda}^{r}\frac{\mu'}{ K}d\tau\\
&\leq&2\delta_{0}^{-2}
\la
-\frac{\mu(r)}{ K(r)}+\frac{\mu(1/\lambda)}{ K(1/\lambda)}-\int^{r}_{1/\lambda} \mu \frac{K'}{K^{2}}d\tau\\
&\leq & (8\de_{0}^{-2}+3\de_{0}^{-4})
\la\ .
\end{eqnarray*}
On the other hand, 
by \eqref{0-eq-tuo}-\eqref{0-eq-1.3},
we have $2\lambda K+\mu
\ge \la K>
\de_{0}\la
$,
and
$$\mu=\frac{-G(r)-\lambda K'-\mu'}{2\lambda K+\mu},$$
then we have 
$$-\int^{r}_{1/\lambda}\mu(\tau)d\tau=\int^{r}_{1/\lambda}\frac{G(\tau)}{2\lambda K+\mu}d\tau+\int^{r}_{1/\lambda}\frac{\lambda K'}{2\lambda K+\mu}d\tau+\int^{r}_{1/\lambda}\frac{\mu'}{2\lambda K+\mu}d\tau\ .$$
Similar arguments yield
\begin{eqnarray*}
-\int^{r}_{1/\lambda}\mu(\tau)d\tau
&\le &
2\de_{0}^{-2}
+\left.\frac{\mu}{2\lambda K+\mu}\right|_{\tau=1/\lambda}^r
+\int^{r}_{1/\lambda}\frac{\mu(2\lambda K'+\mu')}{(2\lambda K+\mu)^{2}}d\tau
\\
&\le& 
8\de_{0}^{-2}+\int^{r}_{1/\lambda}\frac{\mu(\lambda K'+\mu'+G)}{(2\lambda K+\mu)^{2}}d\tau+\int^{r}_{1/\lambda}\frac{\mu(\lambda K'-G)}{(2\lambda K+\mu)^{2}}d\tau\\
&\le & 
8\de_{0}^{-2}+6\de_{0}^{-4}
\ ,
\end{eqnarray*}
where we have used the fact that $\mu(\lambda K'+\mu'+G)\leq 0$ due to \eqref{0-eq-tuo}.

In conclusion, we have proved that
$|\int^{r}_{1/\lambda}\mu(\tau)d\tau|\le 8\de_{0}^{-2}+6\de_{0}^{-4}$.
Thus,  we see that
\beeq
\label{0-eq-1.6}y=e^{\lambda\int^{r}_{1/\lambda}K(\tau)d\tau+\int_{1/\lambda}^{r}\mu(\tau)d\tau}\simeq  e^{\lambda\int^{r}_{1/\lambda}K(\tau)d\tau}\ ,\ r\ge\la^{-1}\ ,\eneq
which 
proves  
\eqref{0.1-eq-1.6} for the case $\ep=y_0=1$ and completes the proof.
\end{proof}

\subsection{Radial 
entire solutions to the elliptic equation $\Delta_{\gm_1}\phi=\lambda^{2}\phi$}
When $\gm=\gm_1$ which is 
spherical symmetric,
we will use the radial solutions as candidate.

In this case, we could determine the asymptotic behavior of radial 
entire solutions.
\begin{lem}
\label{elp-eq}
Let $n\geq 2$, $\gm=\gm_1$,
and $R_2\ge 1$ such that
$n-1-r K'/K\ge 0$ for any $r\ge R_2$ (which is ensured by \eqref{dl2} and \eqref{unelp}).
Then
  there exists $c_2>0$ such that for any $0<\la\le 1/R_2$, there is a radial solution of 
  $\De_{\gm_1} \Phi_\la(x)=\la^2\Phi_\la(x)$,
  such that
\beeq\label{eq-eigfct-est}
c_{2}  \< \la r\>^{-\frac{n-1}{2}}e^{\la \int^{r}_{0}K(\tau)d\tau}
<\Phi_\la(x) 
< c_2^{-1}\< \la r\>^{-\frac{n-1}{2}}e^{\la \int^{r}_{0}K(\tau)d\tau} \ .
\eneq
\end{lem}

Actually, the solution will be constructed such that 
$\Phi_\la(\la^{-1}\omega)=1$, for any $\omega\in \Sp^{n-1}$.
To prove the estimates of $\Phi_\la$, we divide the proof into two parts, inside the ball $B_{\la^{-1}}$ and exterior to the ball.

\subsubsection{Inside the ball $B_{\la^{-1}}$.}
We first consider the Dirichlet problem within $B_{1/\lambda}$
\beeq
\label{Diri}
\begin{cases}
\Delta_{\gm_1}\Phi_\la=\lambda^{2}\Phi_\la, x\in B_{1/\lambda}\\
\Phi_\la|_{\pa B_{1/\lambda}}=1
\end{cases}
\eneq
Then by standard existence theorem of elliptic equation (see, e.g., Evans \cite{Evans10}), there is a unique (hence radial) solution $\Phi_\la\in H^{1}(\bar{B}_{1/\lambda})\cap C^{\infty}(\bar{B}_{1/\lambda})$. In addition, we have $0< \Phi_\la\leq 1$ in $B_{1/\lambda}$. In fact, if there exists $x_{0}\in B_{1/\lambda}$ such that $\Phi_\la(x_{0})\leq 0$, then by strong maximum principle (see, e.g., Evans \cite{Evans10}, page 350, Theorem 4), we get $\Phi_\la\equiv 1$ within $B_{1/\lambda}$, which is a contradiction. By Hopf's lemma (\cite{Evans10}, page 347), we have $\pa_{r}\Phi_\la>0$ for all $r>0$. 

For future reference, we need to obtain more information concerning the behavior of $\Phi_\la$.
At first, 
we  claim that we have the following derivative estimates
\beeq\label{eq-dr-est}
\pa_{r}\Phi_\la(r)\leq D_{0}\la^2 r \Phi_\la (r)\ ,
|\pa_r^2\Phi_\la(r)|\le D_0 \la^2\Phi_\la (r)\ , \forall r\ge 0
\eneq 
for some constant $D_0$ independent of $\lambda\in (0, 1]$.
 In addition,
we
 claim that there exists a uniform lower bound of
 $\Phi_\la$ for $\la \in (0, 1]$:
\beeq
\label{uniform1}
\Phi_\la(x)\geq \Phi_\la(0)\ge C>0, \forall \la\in (0, 1]\ .
\eneq
 
\subsubsection{Derivative estimates 
\eqref{eq-dr-est}}
As $\Phi_\la$ is radial, by \eqref{dl3} we have 
$$\Delta_{\gm_1}\Phi_\la=K^{-1}r^{1-n}\pa_{r}( K^{-1} r^{n-1}\pa_{r}\Phi_\la)\ ,$$
and so
\beeq
\label{s01}
\pa_{r}( K^{-1} r^{n-1}\pa_{r}\Phi_\la)=\lambda^{2}Kr^{n-1}\Phi_\la.
\eneq

As $\Phi_\la$ is increasing, we get
$$
K^{-1} r^{n-1}\pa_{r}\Phi_\la\le
\int_0^r\lambda^{2}\tau^{n-1}K\Phi_\la d\tau
\le \lambda^{2}r^{n}\|K\|_{L^\infty} \Phi_\la(r)\ ,$$
that is, $
\pa_{r}\Phi_{\lambda}
\le\|K\|^2_{L^\infty}\lambda^{2}r\Phi_\la(r)$.

For the second order derivative of $\Phi_{\lambda}$, by \eqref{s01}, we have 
$$|\pa^{2}_{r}\Phi_{\lambda}|
=|\lambda^{2}K^{2}\Phi_\la-\big(\frac{n-1}{r}-\frac{K'}{K}\big)\pa_{r}\Phi_\la|
\leq \lambda^{2}K^{2}\Phi_{\lambda}+\frac{C}{r}\pa_{r}\Phi_{\lambda} \les \lambda^{2}\Phi_{\lambda}$$
for some $C>0$ due to \eqref{dl2}.

\subsubsection{Uniform lower bound: \eqref{uniform1}}
To prove the uniform lower bound \eqref{uniform1}, we use scaling argument so that we could compare all these solutions in the same domain.

Let $f_{\lambda}(x)=\Phi_\la(x/\la)$, then $f_{\lambda}\in (0, 1]$ which satisfies
\begin{align*}
\begin{cases}
\Delta_{\tilde{\gm}_\la}f_{\lambda}=f_{\lambda}, x\in B_{1}\\
f_{\lambda}|_{\pa B_{1}}=1
\end{cases}
\end{align*}
where $\tilde{\gm}_\la(x)=\gm_1(x/\lambda)$. Since $f_{\lambda}$ is radial increasing, we know that
$$\liminf_{0<\lambda\leq 1}\inf_{x\in B_{1}}f_{\lambda}=\liminf_{0<\lambda\leq 1}f_{\lambda}(0):=C\ge 0\ .$$
To complete the proof, we need only to prove
$C>0$.

By definition, there exists a sequence $\lambda_{j}\to 0$ such that $f_{\lambda_{j}}(0)\to C$ as $j\to \infty$.
Recalling \eqref{eq-dr-est}, we know that
\beeq\label{eq-dr-est'}
0<f_\la(r)\le 1,\
\pa_{r}f_\la(r)
\leq D_{0} r f_\la (r)\ ,
|\pa_r^2f_\la(r)|\le D_0 f_\la (r)\ .\eneq
By the Arzela-Ascoli theorem, we see that
 there exists a subsequence of $\lambda_{j}$ (for simplicity we still denote the subsequence as $\lambda_{j}$) such that 
$f_{\lambda_{j}}$ converges uniformly to 
$f$ in $C^{1}(\bar B_{1})$ as $j$ goes to infinity.

In view of the equations satisfied by $f_{\la_j}$, we see that, for any $\phi\in C^\infty_0(B_1)$, we have
\beeq\label{eq-1027}\int_{B_1} (\sum_{k,l}\tilde{\gm}_{\la_j}^{kl}\pa_k f_{\la_j} \pa_l \phi +f_{\la_j} \phi) d V_{\gm_{\la_j}}=0\ .\eneq
By \eqref{dl3}-\eqref{dl2}, we see that $\tilde{\gm}_{\la_j}^{kl}(x)=
\gm_1^{kl}(x/\lambda_{j})\to \delta^{kl}$ as $\lambda_{j}\to 0$ when $x\neq 0$, and
$\tilde{\gm}_{\la_j}^{kl}\in L^\infty$ uniformly.
Then, we could take limit in
\eqref{eq-1027} to obtain
\beeq\label{eq-1027.1}\int_{B_1} (\nabla f \cdot \nabla \phi +f \phi) d x=0\ ,\eneq for any $\phi\in C^\infty_0(B_1)$, which tells us that $f\in C^1$ is a weak solution to the Poisson equation
\begin{align*}
\begin{cases}
\Delta f=f, x\in B_{1}\ ,\\
f|_{\pa B_{1}}=1\ ,
f(0)=C\ .
\end{cases}
\end{align*}
By regularity and
strong maximum principle, we know that $f\in C^\infty(B_1)$ and
$f(0)=C>0$, which completes the proof.

\subsubsection{The solution outside the ball $B_{1/\lambda}$.}
When $r\geq 1/\lambda$, by \eqref{s01}, the equation
$\Delta_{\gm_1}\Phi_\la=\lambda^{2}\Phi_\la$
is reduced to a second order ordinary differential equation
\beeq
\label{qidian}
\begin{cases}
\pa^{2}_{r}\Phi_\la+\big(\frac{n-1}{r}-\frac{K'}{K}\big)\pa_{r}\Phi_\la=\lambda^{2}K^{2}\Phi_\la,\\
\Phi_\la(1/\lambda)=1, \Phi_\la'(1/\lambda)\in (0, D_{0}\lambda].\\
\end{cases}
\eneq

Let $y=r^{\frac{n-1}{2}}K^{-\frac 1 2}\Phi_\la$, then, if
$\la\in (0, 1/R_2)$, we have
 $y>0$ and 
$$y'=\frac{1}{2}r^{\frac{n-3}{2}}K^{-1/2}\Big(n-1-r\frac{K'}{K}\Big)\Phi_\la+r^{\frac{n-1}{2}}K^{-1/2}\Phi_\la'
 \ .$$
Recalling that
we have \eqref{eq-dr-est} and the fact that $n-1-rK'/K\ge 0$ for $r\ge 1/\la\ge R_2$, due to the assumption \eqref{dl2} and \eqref{unelp},
we see that
\beeq\label{eq-y'} 
0< y'\les 
(1/r +\la^2 r) y, \forall r\ge \la^{-1}\ .
\eneq
 Moreover, it turns out that $y$ satisfies
\beeq
\label{1.2}
\begin{cases}
y''-\lambda^{2}K^{2}y+G(r)y=0,\\
y(1/\lambda)=y_{0}, y'(1/\lambda)=y_{1},
\end{cases}
\eneq
where $y_{0}=\lambda^{-\frac{n-1}{2}}K^{-\frac{1}{2}}(1/\lambda)\simeq \la^{-\frac{n-1}{2}}$,
$0<y_1\les \la y_0$, and 
$$
G(r)=-\frac{(n-1)(n-3)}{4r^{2}}+\frac{(n-1)K'}{2rK}+\frac{K''}{2K}-\frac{3}{4}\left(\frac{K'}{K}\right)^{2}\ .$$
By the assumption \eqref{dl2}, we see that
$r^{2}G(r)$ is uniformly bounded and so
\beeq\label{eq-Gest}
\|K'\|_{L^{1}([1/\lambda, \infty))}\les 1\ ,\ 
\|G(r)\|_{L^{1}([1/\lambda, \infty))}\le \la\|  r^{2}G(r)\|_{ L^{\infty}([1/\lambda, \infty))}\les \la\ .\eneq

For \eqref{1.2}, we could apply
Lemma \ref{thm-ode} with $\ep=1$ to conclude
$$y(r)
\simeq  y_0 e^{\lambda\int^{r}_{1/\la }K(\tau)d\tau}\simeq  \la^{-\frac{n-1}{2}} e^{\lambda\int^{r}_{1/\la } K(\tau)d\tau}
\simeq  \la^{-\frac{n-1}{2}} e^{\lambda\int^{r}_{0} K(\tau)d\tau}
 ,\ r\ge 1/\lambda\ .$$
Hence
for any $\la\in (0, R_2^{-1})$,
  we obtain
$$\Phi_\la=r^{-\frac{n-1}{2}}K^{1/2}y\simeq
(r\lambda)^{-\frac{n-1}{2}}e^{\lambda\int^{r}_{0}K(\tau)d\tau}
,\ \forall r\ge \la^{-1}.$$

Recalling \eqref{uniform1}, we have
$$\Phi_\la(r)\simeq 1, \ \forall r\le \la^{-1}, \la\in (0, R_2^{-1}).$$
Combining this two estimates, we 
get
\beeq
\Phi_\la
\simeq
\<r\lambda\>^{-\frac{n-1}{2}}e^{\lambda\int^{r}_{0}K(\tau)d\tau}
,\ \forall 
\la\in (0, R_2^{-1})
,\eneq
which
completes the proof of \eqref{eq-eigfct-est}.

\subsection{Proof of Lemma \ref{elp}}
With help of Lemma \ref{elp-eq},
 due to the exponential feature of the short range perturbation $\gm_2$, a perturbation argument will give the proof of  Lemma \ref{elp}.
Here, we basically follow the similar proof as in Wakasa-Yordanov \cite[Lemma 2.2]{WaYo18-1pub}.

Let $\psi=\phi_\la-\Phi_\la$ where $\Delta_{\gm_{1}}\Phi_\la=\lambda^{2}\Phi_\la$. 
Then with Lemma \ref{elp-eq} in hand, we are reduced to prove existence of $\psi$ and show smallness of $\|\psi\|_{L^{\infty}}$. In fact,  $\psi$ satisfies
\begin{align}
\label{est-yx}
\Delta_{\gm}\psi=\Delta_{\gm}(\phi_\la-\Phi_\la)=\lambda^{2}\phi_\la-\lambda^{2}\Phi_\la+(\Delta_{\gm_{1}}-\Delta_{\gm})\Phi_\la=\lambda^{2}\psi+W\Phi_\la.
\end{align}
We claim that, for any $\la\in (0, \min(\al/2, 1/R_2))$,
\beeq
\label{rem1}
\|W\Phi_\la\|_{ L^{q}}\les \la^2,\ \forall \ q\geq 1\ .
\eneq
Actually, we have
$$W\Phi_\la=
(\gm^{jk}_{1}-\gm^{jk})
\pa_{j}\pa_{k}\Phi_\la+
[g_1^{-1/2}\pa_j (g_1^{1/2}\gm_1^{jk})
-g^{-1/2}\pa_j (g^{1/2}\gm^{jk})]
 \pa_k\Phi_\la\ .
$$
By  
\eqref{eq-eigfct-est}, \eqref{eq-dr-est}
 and \eqref{dlfjia} 
 with $\lambda<\min(\al/2, 1/R_2)$, it is easy to see that
$$|W \Phi_\la|\les \sum_{j,k}
\sum_{1\le |\be|\leq 2}
|\nabla^{\leq 1}g^{jk}_{2}||\nabla^{\be}\Phi_\la|\les
\la^2\< r\> e^{(\la-\al)\int^{r}_{0}K(\tau)d\tau}
\les
\la^2 \< r\>e^{-\frac{\al}{2}\de_0 r}
\ ,$$
which gives us \eqref{rem1}.
  
Standard elliptic theory ensures that there exists a unique weak $H^1$ solution $\psi$
to \eqref{est-yx}.
  To show the smallness of $\| \psi\|_{L^{\infty}}$, we first take the (natural) inner product of \eqref{est-yx} with $\psi$ to get
$$-\langle \Delta_{\gm}\psi, \psi\rangle_\gm+\lambda^{2}\langle \psi, \psi\rangle_\gm=-\langle W\Phi_\la,  \psi\rangle_\gm\ .$$
Thus by the Cauchy-Schwarz inequality and uniform elliptic condition \eqref{unelp} we have 
$$\delta_{0}\|\nabla \psi\|^2_{L^{2}_{\gm}}+\lambda^{2}\|\psi\|^2_{L^{2}_{\gm}}\les \lambda^{2}\|\psi\|_{L^{2}_{\gm}},$$
which yields
$$\|\psi\|_{L^{2}_x}\les 1, \ \|\nabla \psi\|_{L^{2}_{x}}\les \lambda\ .$$
By applying Gagliardo-Nirenberg inequality, we get 
$$\|\psi\|_{L^{q_0}_{x}}\leq C\lambda^{\theta},$$
for some fixed $q_0>2$ with $\theta=n(1/2-1/q_0)\in (0, 1)$. Thus by
local properties of weak solutions for uniformly elliptic equations, see
Gilbarg-Trudinger
\cite[Theorem 8.17]{MR1814364}, we know that
there exist a uniform constant $C>0$ (depending only on $\gm$, $q_0$, $n$), such that for any $x\in \R^n$,
we have
$$
\|\psi\|_{L^{\infty}}=
\|\psi\|_{L^\infty_x( L^{\infty}(B_1(x)))}\le C(\|\psi\|_{L^\infty_x (L^{q_0}(B_2(x)))}
+\|W\Phi_\la\|_{L^{n}})\les \la^\theta\ ,
$$
which completes the proof.

\section{Blow up in subcritical case}\label{sec-sub}
In the subcritical case, $p<p_c(n)$, 
as has been classical in this field, we could use the test function method to give a proof of 
Theorem \ref{thm-1-AE}, as soon as we 
obtained the solution of \eqref{1.1} with the desired property \eqref{1.50}, for some fixed $\la>0$.

In the presence of damping lower order term, with $b(t)\in L^1$, we propose the approach of change of variable for Theorem \ref{thm-2-AE-damp}, which essentially reduced the problem to the standard problem. This simplifies the corresponding proof in Lai-Takamura\cite{LaiTa18} for the subcritical case with $b\ge 0$ and $\gm=\gm_0$. 

As is standard, when we employ the test function method, we typically need  the Kato type lemma to conclude nonexistence of global solutions as well as the upper bound of the lifespan, for which proof, we refer Sideris \cite{Sideris84} for the blow up result and Zhou-Han \cite{ZhouHan11} for the upper bound. 
\begin{lem}[Kato type lemma]
\label{lem4.1}
Let $\be>1$, $a\geq 1$ and $(\be-1)a>\al -2$. If $F\in C^{2}([0, T))$ satisfies
$$F(t)\geq \delta (t+1)^{a},\ F''(t)\geq k(t+1)^{-\al}F^{\be}\ ,$$
for some positive constants $\delta, k$. Then $F(t)$ will blow up at finite time, that is, $T<\infty$. Moreover, we have the following upper bound of $T$,
$$T\leq c\delta^{-\frac{\be-1}{(\be-1)a-\al+2}},$$
for some constant $c$ which is independent of $\delta$. 
\end{lem}

In the following, we give a unified proof of Theorems \ref{thm-1-AE},
\ref{thm-2-AE-damp}, in the subcritical case.

\subsection{Change of variable}\label{sec-change}
Considering \eqref{2.1'} with $b(t)\in L^{1}$, we introduce the new variable
\beeq
 \label{bianhuan}
 m(t)=e^{\int^{t}_{0}b(\tau) d\tau}, \quad s=\int^{t}_{0}\frac{1}{m(\tau)}d\tau=h(t)\ .
 \eneq
Then it is easy to see $\pa_{s}=m(t)\pa_{t}$ and
 $$(\pa^{2}_{t}+b(t)\pa_{t})u=m^{-1}\pa_{t}(m\pa_{t}u)=m^{-2}m\pa_{t}(m\pa_{t})=m^{-2}(m\pa_{t})^{2}u=m^{-2}(t)\pa^{2}_{s}u.$$
Let $t=\eta(s)$ denote the inverse function of $h$. 
Thus \eqref{2.1'} becomes
\beeq
\label{eq-2.1}
\pa^{2}_{s}u-\tilde{m}^{2}(s)\Delta_{\gm}u=\tilde{m}^{2}(s)|u|^{p}, \quad \tilde{m}(s)=m(\eta(s))\ .
\eneq
Traditionally, we still use $t$ to denote time, thus we are reduced to consider the following equation
\beeq
\label{bianhuan1}
\pa^{2}_{t}u-\tilde{m}^{2}(t)\Delta_{\gm}u=\tilde{m}^{2}(t)|u|^{p}, 
\ \tilde{m}(t)=m(\eta(t))\ ,
\eneq
with initial data
$$u(0,x)=\ep u_0(x), u_{t}(0,x)=\ep u_1(x)\ .$$
Thus by \eqref{fspyl}, the support of solution $u$ of \eqref{bianhuan1} satisfies
\beeq
\label{supp01}
\supp u\subset \{(t, x); \int^{|x|}_{0}K(\tau)d\tau\leq \eta(t)+R_{1}\}\ .
\eneq
Furthermore, there exists $\delta_{1}\in (0, 1)$ such that $\tilde{m}(t)\in [\delta_{1}, 1/\delta_{1}]$ since $b(t)\in L^{1}$.

For future reference, we record that $\eta'(t)=1/h'(\eta(t))=m(\eta(t))\in [\delta_{1}, 1/\delta_{1}]$, $\eta(0)=0$, $m'=bm$,  $m(0)=1$,
$\tilde m'(t)=m'(\eta(t))\eta'(t)=b(\eta(t))m^2(\eta(t))=b(\eta(t))\tilde m^2(t)$, that is
\beeq\label{eq-change}
\eta'(t)=m(\eta(t))=\tilde m(t),\ 
\tilde m'(t)=b(\eta(t))\tilde m^2(t),\ 
\eta(0)=0, m(0)=\tilde m(0)=1\ .
\eneq

\subsection{Proof of the subcritical blow up}
Recall that
for the standard wave operator $\pt^2-\Delta_g$, the test function is typically
chosen to be $\psi(t, x)=e^{-\la t}\phi_\la (x)$, which solves linear homogeneous wave equation. In the presence of $\tilde m$, we try choosing similar form of test function, which may not be the exact solution to 
 the linear homogeneous wave equation.
 
Let $\psi(t, x)=e^{-\la_1\eta(t)}\phi(x)$ be the test function, 
 where $\phi=\phi_{\la_1}$ is the solution of \eqref{1.1} with $\lambda=\la_1\in (0, \min(1,\la_0))$.

For given nontrivial data $(u_0, u_1)$ with  \eqref{hs2} and sufficiently small
 $\ep>0$, we know from Lemma \ref{thm-lwp} that there is a local (weak) solution $u$ satisfying \eqref{supp01} for
 \eqref{bianhuan1}. By continuity, there exists a maximal time of existence $[0, T_\ep)$. 
Without loss of generality, we assume $T_\ep>2$. 
 
Let $F(t)=\int u(t,x)dv_{\gm}$, $H(t)=\int u\psi dv_{\gm}$. Then
we claim that
\beeq 
\label{jian1}
F''\gtrsim |F|^{p}(1+t)^{-n(p-1)}, 
F''\gtrsim |H|^{p}(1+t)^{(n-1)(1-p/2)}, \forall t\in [0, T_\ep)\ ,
\eneq
\beeq\label{eq-lb2}
 H\gtrsim \ep\ , \forall t\in [1, T_\ep)\ .
\eneq
 We postpone the proof to the end of this section.

Since $F(0)=\int_{\R^{n}} u_0d v_{\gm}\ge 0$, $F'(0)=\int_{\R^{n}}u_1 d v_{\gm}\geq 0$, then by \eqref{jian1}, we see that $F''(t)\geq 0$ for any $t\geq 0$, thus $F'(t)\geq 0$ and $F(t)\geq 0$ for any $t\geq 0$.
Then,
with help of 
\eqref{eq-lb2}, 
as $(n-1)(1-p/2)>-1$,
we could integrate the 
second  inequality in \eqref{jian1}
 twice to get
\beeq\label{eq-lb} F\gtrsim \ep^{p}(1+t)^{2+(n-1)(1-p/2)},\ \forall t\in [2, T_\ep)\ .\eneq
Heuristically, 
the lower bound \eqref{eq-lb}
 could be improved with help of the first inequality  in \eqref{jian1}, if
$$(2+(n-1)(1-p/2))p-n(p-1)+2>(2+(n-1)(1-p/2))\ ,$$
that is $p<p_c(n)$, which suggests blow up results.
Actually, for $1<p<p_c(n)$, we could apply
Lemma \ref{lem4.1} with
\eqref{jian1} and 
\eqref{eq-lb} for $t\ge 2$,
to obtain the blow up results  and 
the upper bound of lifespan $T_\ep$ as in 
\eqref{eq-life}.

Finally, we give the proof of \eqref{eq-life-special},
under the additional assumption that
$u_1$ does not vanish identically.
Actually, by the assumption on the data, we know that
$F(0)=\ep\int u_0 d v_{\gm}\ge 0$,
 $F'(0)=\ep\int u_1 d v_{\gm}\ge c\ep$ for some $c>0$.
Recall that the first inequality in \eqref{jian1} tells us that
$F''\geq 0$,
 we have 
\beeq\label{eq-lb3} F(t)\geq c\ep t,\ \forall t\ge 0\ ,\eneq
which improves the lower bound of $F$ in size, but not on rate, than that in \eqref{eq-lb}. Applying 
 Lemma \ref{lem4.1}
with 
\eqref{jian1} and 
\eqref{eq-lb3} for $t\ge 1$,
when $p\in (1, 1+2/(n-1))$,
we obtain 
$$T_\ep\le C_0 \ep^{-\frac{(p-1)}{2-(n-1)(p-1)}}\ ,
$$
 which is better than
that in \eqref{eq-life}, exactly when $n=2$ and $1<p<2$. 
This completes the proof of
\eqref{eq-life-special}.

\subsection{Proof of  \eqref{jian1}-\eqref{eq-lb2}}
\subsubsection{First inequality of  \eqref{jian1}}
Let $u$ be solution to
 \eqref{bianhuan1}  satisfying \eqref{supp01}, we have
$$\int_{\R^{n}}\pa^{2}_{t}u \ d v_{\gm}-\tilde{m}^{2}(t)\int_{\R^{n}}\sqrt{|\gm|}^{-1}\pa_{i}(g^{ij}\sqrt{|\gm|}\pa_{j}u)d v_{\gm}=\tilde{m}^{2}(t)\int_{\R^{n}}|u|^{p}d v_{\gm}\ .$$
As $F(t)=\int u(t,x)d v_{\gm}$ and $u$ vanish for large $x$, 
the second term vanishes and so
\beeq
\label{app10}
F''=\tilde{m}^{2}(t)\int_{\R^{n}}|u|^{p}d v_{\gm}.
\eneq

Recall that $\eta'(t)=1/h'(\eta(t))=m(\eta(t))\in [\delta_{1}, 1/\delta_{1}]$ and $\eta(0)=0$, we have 
$$\eta(t)=\int^{t}_{0}\eta'(s)ds\leq t/\delta_{1}\ .$$
Then by \eqref{supp01} and $K\in[\delta_{0}, 1/\delta_{0}]$, we see that $$\supp u\subset\{(t, x); \int^{|x|}_{0}K(\tau)d\tau\leq \eta(t)+R_{1}\}\subset\{(t, x); |x|\leq t/(\delta_{0}\delta_{1})+R_{1}/\delta_{0}\}\ .$$ 
By H\"{o}lder's inequality we get 
$$|F|\leq \int\limits_{|x|\leq \frac{t}{\delta_{0}\delta_{1}}+\frac{R_{1}}{\delta_{0}}}|u|d v_{\gm}\les \Big(\int |u|^{p}d v_{\gm}\Big)^{1/p}\Big( \frac{t}{\delta_{0}\delta_{1}}+\frac{R_{1}}{\delta_{0}}\Big)^{n/p'}.$$
Recalling \eqref{app10},
we have
$$F''=\tilde{m}^{2}(t)\int|u|^{p}d v_{\gm}\gtrsim |F|^{p}\langle t\rangle^{-n(p-1)},$$
which gives us the first inequality in
\eqref{jian1}.

\subsubsection{Second inequality of  \eqref{jian1}}
To relate $H$ with $F''$, we
use \eqref{supp01}  and H\"{o}lder's inequality to obtain
$$
|H(t)|=|\int u(t) \psi(t) d v_{\gm}|\les \Big(\int |u(t)|^{p} d v_{\gm}\Big)^{1/p}\Big(\int_{A_1}\psi(t, x)^{p'} dx\Big)^{1/p'}\ ,$$
where
$A_1:=\{x\in \R^n: \int^{|x|}_{0}K(\tau)d\tau\leq \eta(t)+R_{1}\}$.
To control the last term, we use \eqref{1.50} with $\lambda=\la_1$ and change of variable to get
\begin{eqnarray*}
\int_{A_1}\psi(t, x)^{p'} dx
&\les &\int_{A_1} \<\la_1 x\>^{-\frac{(n-1)p'}{2}}e^{p'\la_1(\int^{|x|}_{0}K(\tau)d\tau-\eta(t))}dx\\
&\les& \int_{0}^{\la_1(\eta(t)+R_{1})}e^{p'(r-\la_1\eta(t))}(1+r)^{n-1-\frac{(n-1)p'}{2}}dr\\
&\les& (1+\la_1\eta(t))^{n-1-(n-1)p'/2}
\les \<t\>^{n-1-(n-1)p'/2}
\ .
\end{eqnarray*}
Then by \eqref{app10}, we get
$$F''
\gtrsim \int |u|^{p} d v_{\gm} \gtrsim |H|^{p}
\<t\>^{-\frac{p}{p'}\big(n-1-(n-1)p'/2\big)} \gtrsim |H|^{p}
(1+t)^{(n-1)(1-p/2)}\ ,$$
which is the second inequality in
\eqref{jian1}.

\subsubsection{Proof of \eqref{eq-lb2}}
Let $G(t)=
 \int u(t) \phi(x)dv_\gm
 =e^{\la_1\eta(t)}H$.
 With $\phi(x)=\phi_{\la_1}(x)$ as the test function, 
we obtain by integration that
$$ G''-\tilde{m}^{2}(t)\la_1^2 G=\int \tilde{m}^{2}(t)|u|^{p}\phi dv_\gm\ge 0 \ . $$
Recall that
$\int^{t}_{0}\tilde{m}d\tau=\eta(t)$,
$$G(0)=\ep\int u_0 \phi d v_{\gm}, \ G'(0)=\ep\int u_1\phi d v_{\gm}\ .$$
We claim that
\beeq\label{eq-Lev}G(t)\geq C'\ep e^{\la_1\eta(t)}, t\geq 1\ ,
\eneq
for some $C'>0$, which gives us the desired
lower bound of $H(t)$ when $t\geq 1$, and so 
is \eqref{jian1}.

It remains to prove \eqref{eq-Lev}.  
By the condition \eqref{hs2} on the data, we have  
$$G(0)= C_{0}\ep, \ G'(0)= C_{1} \ep,$$
for some $C_{0}, C_{1}\ge 0$ with $C_{0}+C_{1}>0$. It is clear that $G(t)\ge C_{0}\ep$, $G''\ge \de_1^2\la_1^2 C_{0}\ep$ for all $t\ge 0$,
and $G'
\ge (C_{1}+t\de_1^2\la_1^2 C_{0})\ep
>0$ for all $t> 0$. Moreover, it is bounded from below by
$\ep y$, which is solution to
\beeq\label{eq-Lev2}y''-\tilde{m}^{2}(t)\la_1^2 y=0,
y(0)=G(0)/\ep=C_{0}, y'(0)= G'(0)/\ep=C_{1}\ .\eneq
For \eqref{eq-Lev2}, we could apply
Lemma \ref{thm-ode} together with elementary ODE to give the proof. Here, as it has a fixed parameter $\la_1$, we present another proof by 
 applying the classical Levinson theorem (see, e.g., 
\cite[Chapter 3, Theorem 8.1]{CoLe55}).
 As is clear, for the proof of subcritical results, we could avoid
 Lemma \ref{thm-ode}, but relying on the Levinson theorem.

For $y$, it is clear from continuity that we have
$y,  y'> 0$ for all $t> 0$.
More precisely, we have
$y(t)\ge C_{0}$, $y''\ge \de_1^2\la_1^2 C_{0}$ for all $t\ge 0$,
and so \beeq\label{eq-lb4}
y'\ge C_{1}+t\de_1^2\la_1^2 C_{0}
>0\ ,\ y\ge C_0+C_1 t>0,\ \forall t> 0\ .\eneq
Let $Y(t)=(Y_{1}(t), Y_{2}(t))^{T}$ with $Y_{1}=y$, $Y_{2}=y'$, then we have 
$$Y'=(A+V(t))Y,
A=\left(\begin{array}{cc}0 & 1 \\ \lambda_1^{2}k^2 & 0\end{array}\right),
V(t)=\left(\begin{array}{cc}0 & 0 \\ \lambda_1^{2}(\tilde{m}^{2}(t)-k^2) & 0\end{array}\right),$$
where
$k=\tilde{m}(\infty)=\exp(\int_0^\infty b(t)dt)$.
Noticing that $V'\in L^1$ and $\lim_{t\to\infty} V(t)=0$, we could apply the Levinson theorem to the system. Then there exists $t_0\in [0,\infty)$ so that we have two independent solutions, which have the asymptotic form
as $t\to \infty$
$$Y_{\pm}(t)=
\left(\begin{array}{c}1+o(1) \\ 
\pm\lambda_1 k+o(1)\end{array}\right)
e^{\pm\lambda_1\int^{t}_{t_0}\tilde{m}(\tau)d\tau}
$$
Thus, we have, for some $c_1, c_2$,
\beeq
y=c_1(1+o(1))
e^{\lambda_1\int^{t}_{t_0}\tilde{m}(\tau)d\tau}
+c_2(1+o(1))
e^{-\lambda_1\int^{t}_{t_0}\tilde{m}(\tau)d\tau}\eneq
as $t\to\infty$.
By \eqref{eq-lb4}, we know that
$y\ge C_{0}+C_{1}$
 for all $t\ge 1$, it is clear that $c_1>0$.
Then there exists some $T\ge 1$ such that
\beeq
y\simeq  c_1 
e^{\lambda_1\int^{t}_{t_0}\tilde{m}(\tau)d\tau}
\simeq  c_1 e^{\lambda_1 \eta(t)}
\ ,\ \forall t\ge T\ . \eneq
Combining it with
\eqref{eq-lb4}, we see that there exists $C'>0$ such that
$$y(t)\ge C'
e^{\lambda_1 \eta(t)}
\ , \forall t\ge 1\ .
$$
In conclusion, 
we obtain
$$G\ge \ep y\ge C'\ep e^{\lambda_1 \eta(t)}\ , \forall t\ge 1\ .$$
which completes the proof of
 \eqref{eq-Lev}.

\section{The Blow up of semilinear wave equations in critical case}
\label{sec-crit}
In this section, 
we give the proof of the critical case for 
Theorems \ref{thm-1-AE} and \ref{thm-2-AE-damp}.

With help of the solutions of \eqref{1.1} with the desired property \eqref{1.50}, for any $\la\in (0,\la_0]$, 
the proof of Theorem \ref{thm-1-AE} in the critical case could be given by following the similar proof as 
that in \cite{WaYo18-1pub}, for the case $b=0$ and $\gm=\gm_0+\gm_2$.

Similar to Section \ref{sec-sub},
we will use the approach of change of variable,
as presented in subsection \ref{sec-change}, for Theorem \ref{thm-2-AE-damp},
when there is a lower order term $b(t)\in L^1$.

\subsection{Preparation}
To conduct the test function method, when $b\neq 0$, we need to solve and determine asymptotic behavior for
certain second order differential equation with parameter $\la\in (0,\la_0]$,
\beeq \label{ode1}
y''-\lambda^{2}\tilde{m}^{2}(t)y=0, y(0)=0, y'(0)=1
\ .\eneq
The following lemma gives us the desired result.
\begin{lem}
\label{Le3.1}
There exists a uniform constant $c_3>0$, which is independent of $\la\in (0,\la_0]$, such that
the solution to the ordinary differential equation \eqref{ode1}, with parameter
 $\la\in (0,\la_0]$,
satisfies
\beeq
\label{eq-lower}
 y\ge c_3 \frac{\sinh \lambda\eta(t)}{\lambda}, \  y'\ge c_3 \cosh \lambda\eta(t)\ .
\eneq
\end{lem}
\begin{proof}
We may want to apply the Levinson theorem to give the proof, similar to that for \eqref{eq-Lev}. However, such a proof may not be sufficient to ensure the uniform constant. Instead, we give a direct proof.

At first, it is clear that $y, y''>0$ for all $t>0$, $y'>0$ for all $t\ge 0$, and so all of these functions are increasing.
Since $\tilde{m}\in [\delta_{1}, 1/\delta_{1}]$, by \eqref{ode1}, we have 
$$\lambda^{2}\delta^{2}_{1}y\leq y''\leq \frac{\lambda^{2}}{\delta_{1}^{2}}y \ ,$$ which gives us
\beeq
\label{eq-4.1}
 \frac{1}{\lambda\delta_{1}}\sinh \lambda \delta_{1}t\leq y\leq \frac{\delta_{1}}{\lambda}\sinh \frac{\lambda}{\delta_{1}}t\ ,\ 
 \cosh \lambda\delta_{1}t\le y'\le
\cosh \frac{\lambda}{\delta_{1}}t
  \ .
 \eneq
We claim that there exists a $\theta>0$ such that when $\lambda t<\theta$, we have 
\beeq\label{eq-claim}\frac{\sinh \lambda \delta_{1}t}{\delta_{1}\sinh \lambda\eta(t)}\geq \de _1^{2}\ .\eneq
Actually,
there exists $\ep>0$ such that
$$
t\le \sinh t
\le \de_1^{-1} t, 
1\le \cosh t
\le \de_1^{-1},
\forall t\in [0, \ep]\ .$$
Let $\theta=\de_1 \ep$, 
as 
$\eta(t)\le \de_1^{-1} t$,
we have $ \lambda\eta(t)\le
\de_1^{-1} t\la\le
\de_1^{-1}\theta=\ep
$ for $\lambda t<\theta$,
and so
$$\frac{\sinh \lambda \delta_{1}t}{\delta_{1}\sinh \lambda\eta(t)}\geq 
\frac{\lambda \delta_{1}t}{ \lambda\eta(t)}\geq 
\de_1^2\ ,
$$
which verifies the claim \eqref{eq-claim}.
Thus we get
$$y
\ge \frac{1}{\lambda\delta_{1}}\sinh \lambda \delta_{1}t
\geq \de_1^2
\frac{\sinh \lambda\eta(t)}{\lambda}, \forall \lambda t\leq \theta\ .$$
Also we have
$$y'\ge 1\ge
\de_1\cosh \lambda\eta(t), \forall \lambda t\leq \theta\ ,$$
which gives us
\eqref{eq-lower} with $c_3\in (0, \de_1^2]$ for 
$\lambda t\leq \theta$.

When $\lambda t\geq \theta$, consider the following equation
\beeq
\label{eq-3.1}
\begin{cases}
y''-\lambda^{2}\tilde{m}^{2}(t)y=0,\\
y(\theta/\lambda)=y_{0}, y'(\theta/\lambda)=y_{1}\ ,
\end{cases}
\eneq
where 
$y_0\sim \la^{-1}$, and $y_1\sim 1$.
For \eqref{eq-3.1}, we could apply
Lemma \ref{thm-ode} with $G=0$ to conclude
$$y(t)
\simeq  y_0 e^{\lambda\int^{t}_{\theta/\la }\tilde m(\tau)d\tau}\simeq  \frac{ e^{\lambda\int^{t}_{\theta/\la }\tilde m(\tau)d\tau}}{\la},\ t\ge \theta/\lambda\ ,$$
$$y'\simeq y_1+y_0\la (e^{\lambda\int^{t}_{\theta/\la }\tilde m(\tau)d\tau}-1)
\simeq e^{\lambda\int^{t}_{\theta/\la }\tilde m(\tau)d\tau},\ t\ge \theta/\lambda\ .
$$
Recall from
\eqref{eq-change}, we have $\int_0^t\tilde{m}(\tau)d\tau=\eta(t)$,
$\la \int_0^{\theta/\la} \tilde{m}(\tau)d\tau\in (\theta\de_0,
\theta\de_0^{-1}),$  and thus
\beeq
y\simeq \la^{-1}e^{\lambda\eta(t) }, \ 
y'\simeq e^{\lambda\eta(t) },\ t\ge \theta/\lambda\ ,
\eneq
which gives us
\eqref{eq-lower}  for 
$\lambda t\ge \theta$ and completes the proof.
\end{proof}

\subsection{Test functions and proof of blow up}
With help of 
subsection \ref{sec-change},
Lemma \ref{elp} and
Lemma \ref{Le3.1}, we could 
follow the similar proof as
 in \cite{WaYo18-1pub} to give the proof of the critical case for 
Theorems \ref{thm-1-AE} and \ref{thm-2-AE-damp}.

At first, we observe that 
to prove the 
$T_\ep \le \exp(C_0\ep^{-p(p-1)})$ for the original equation  \eqref{2.1'},
we need only to prove
 that the lifespan for 
\eqref{eq-2.1}, denoted by $S_\ep$,
satisfies
\beeq\label{eq-life-afterchange} S_\ep \le \exp(C\ep^{-p(p-1)})\eneq
for some other constant $C>0$.
Actually, by \eqref{eq-life-afterchange}, we know that
$$T_\ep\le \eta(S_\ep)\le \eta(\exp(C\ep^{-p(p-1)}))\le
\de_1^{-1}\exp(C\ep^{-p(p-1)})\le\exp(2C\ep^{-p(p-1)})
$$ for sufficiently small $\ep>0$.

In the following, we want to show
blow up and \eqref{eq-life-afterchange} for
 \eqref{bianhuan1}.
 
As in \cite{WaYo18-1pub}, the proof combines several classical ideas in order to generalize and simplify the methods of Zhou \cite{Zh07} and Zhou-Han 
\cite{ZhouHan14criti}. The test function with
certain behavior at future timelike infinity
is constructed
based on the  exponential ``eigenfunctions" of the Laplace-Beltrami operator for asymptotically Euclidean manifolds.
To improve the lower bound of certain auxiliary function, the method of iteration (slicing method) in 
  Agemi-Kurokawa-Takamura \cite{MR1785116} is used.

\subsubsection{Basic ingredients of test functions}
Based on Lemma \ref{elp} and
Lemma \ref{Le3.1},
we are ready to construct 
 basic ingredients of test functions.

 By Lemma \ref{elp}, for any $\la\in (0,\la_0]$,
there exists $\phi_\la$ solving the elliptic eigenvalue problem posed on asymptotically Euclidean manifolds \eqref{dl1}-\eqref{unelp}
$$\Delta_{\gm}\phi_{\lambda}=\lambda^{2}\phi_{\lambda}, $$ 
satisfying \eqref{1.50}, that is,
\beeq
\label{jian2}
c_1< \phi_{\lambda}< c_1^{-1} \langle \lambda x\rangle^{-\frac{n-1}{2}}e^{\lambda\int^{|x|}_{0}K(\tau)d\tau}\ .
\eneq

Fix $T>0$, and consider $t\in [0, T]$.
Let $y_{T,\la}(t)$ be the solution to 
\beeq \label{ode2}
\pt^2 y-\lambda^{2}\tilde{m}^{2}(t)y=0, y(T)=0, y'(T)=-1
\ .\eneq
With $s=T-t$, $z_T(s)=y_{T,\la}(T-s)$, we see that
\beeq \label{ode3}
\pa_s^2z-\lambda^{2}\tilde{m}^{2}(T-s)y=0, z(0)=0, z'(0)=1
\ .\eneq
Then by Lemma \ref{Le3.1} and \eqref{eq-change} we have 
$\int_0^s\tilde{m}(T-\tau) d\tau=\eta(T)-\eta(T-s)$, and
$$y_{T,\la}(t) \gtrsim \frac{\sinh \lambda (\eta(T)-\eta(t))}{\lambda},\ -y_{T,\la}'(t)\gtrsim \cosh \lambda (\eta(T)-\eta(t))\ .$$

Let
$\psi_{T,\la}(t, x)=y_{T,\la}(t)\phi_{\lambda}(x)$, then it is a solution of the linear homogeneous wave equations for $t\in [0, T]$
\beeq
\label{eq-6.1}
(\pa^{2}_{t}-\tilde{m}^{2}(t)\Delta_{\gm})\psi_{T,\la}(t, x)=0,\
\psi_{T,\la}(T, x)=0\ , \ \pt \psi_{T,\la}(t, x)(T)=-\phi_\la\ .
\eneq
In addition, $\psi_{T,\la}(t, x)$ satisfies 
\begin{align}
\label{eq-6.2}
\begin{cases}
\psi_{T,\la}(0, x)=\phi_{\lambda}(x)y_{T,\la}(0)\gtrsim \phi_{\lambda}(x)\frac{\sinh \lambda\eta(T)}{\lambda}
\ge 0
,\\
-\pt \psi_{T,\la}(0, x)=-\phi_{\lambda}(x)y_{T,\la}'(0)\gtrsim \phi_{\lambda}(x)\cosh \lambda \eta(T)\ge 0\ .
\end{cases}
\end{align}

\subsubsection{Integral inequality}

Let $u$ be the solution to
 \eqref{bianhuan1} in $[0, T]$ and
$L=\pa^{2}_{t}-\tilde{m}^{2}(t)\Delta_{\gm}$ be the corresponding wave operator, we use Green's identity and \eqref{eq-6.1} to get
 \begin{eqnarray*}
\int_0^T \int_{\R^n}
\tilde{m}^2(t)\psi_{T,\la} |u|^p dv_\gm dt
&=&\int_0^T \int_{\R^n}
(\psi_{T,\la} L u-u L\psi_{T,\la}) dv_\gm dt\\
& = & \int_{\R^n}
(\psi_{T,\la} \pt  u-u \pt \psi_{T,\la}) dv_\gm |_{t=0}^T \\
 & = &
  \int_{\R^n}
u(T)\phi_{\la} dv_\gm  + \ep \int_{\R^n}
(u_0 \pt \psi_{T,\la}(0)-u_1 \psi_{T,\la}(0)) dv_\gm\ .
\end{eqnarray*}

Assuming 
the data are nontrivial satisfying \eqref{hs2}, by applying \eqref{eq-6.2} with the fact $\tilde m\ge\de_1>0$, we have 
\beeq\int_{\R^{n}}u(T)\phi_{\lambda}d v_{\gm}\gtrsim\int^{T}_{0}\int_{\R^{n}}|u|^{p}\phi_{\lambda}(x)
\frac{\sinh \lambda (\eta(T)-\eta(t))}{\lambda}
 d v_{\gm}dt\ .\label{eq-6.3}
\eneq
 Notice that the left side is basically $G(T)$ in Section \ref{sec-sub}, with small parameter $\la$. 

\subsubsection{Test functions} 
Observing the exponential increasing feature (
$e^{\lambda \eta(T)}$) on the right hand side of  \eqref{eq-6.3}, we 
 multiply $e^{-\lambda(\eta(T)+R_{1})}\lambda^{q}$,
  with $q>-1$, to \eqref{eq-6.3} and integrating from $0$ to $\lambda_{0}$ to get
\begin{eqnarray*}
\nonumber
F(T)&:=&
\int_{\R^{n}}u(T) \int_0^{\la_0} e^{-\lambda(\eta(T)+R_{1})}\phi_{\lambda} \lambda^{q}d\la
d v_{\gm}\\
&\gtrsim&\int^{T}_{0}\int_{\R^{n}}|u|^{p}
\int_0^{\la_0} 
\frac{\sinh \lambda (\eta(T)-\eta(t))}{\lambda}e^{-\lambda(\eta(T)+R_{1})}\phi_{\lambda} \lambda^{q}
 d\la
 d v_{\gm}dt\ .
\end{eqnarray*}
Observing that, due to the support property
\eqref{supp01}
and \eqref{jian2},
$u(T)\in \dot H^{1/2}_{comp}\subset L^1$.
$\phi_\la e^{-\lambda(\eta(T)+R_{1})}=\mathcal{O}(1)$,
$\la^{q}\in L^1_{loc}$, $F(T)$ is well-defined in $T\in [0, S_\ep)$.

To simplify the exposition, we introduce
\beeq\xi_{q}( x, T, t)=\left\{
\begin{array}{ ll}
\int_0^{\la_0} 
\frac{\sinh \lambda (\eta(T)-\eta(t))}{\lambda(T-t)}e^{-\lambda(\eta(T)+R_{1})}\phi_{\lambda} \lambda^{q}
 d\la
      &    t<T\ ,\\
\int_0^{\la_0}  e^{-\lambda(\eta(T)+R_{1})}\phi_{\lambda} \lambda^{q}
 d\la
      &   t=T\ ,
\end{array}
 \right.
\eneq
and so $F(T)=\int_{\R^{n}}u(x, T)\xi_{q}(x, T, T)d v_{\gm}$,
\beeq\label{eq-F-int}
F(T)\gtrsim\int^{T}_{0}(T-t)\int |u(x, t)|^{p}\xi_{q}(x,T, t)dv_{\gm}dt\ .\eneq

\subsubsection{Estimates of $\xi_{q}$}
\begin{lem}
\label{est-xi}
Let $n\geq 2$.
There exist positive constants $A_{j}$, $j=1, 2$, which are independent of $\lambda$, so that
 we have the following estimates:
\begin{enumerate}
  \item if $q> 0$, 
  $0\leq t<T$, then
\beeq
\label{5.4-2}
\xi_{q}(x, T, t)\geq A_{1}\langle T\rangle^{-1}\langle t\rangle^{-q}\ ,
\eneq
  \item if $q>(n-3)/2$, $\int^{|x|}_{0}K(\tau)d\tau\leq \eta(T)+R_{1}$, then
\beeq
\label{5.4-3}
\xi_{q}(x, T, T)\leq A_{2}\langle T\rangle^{-(n-1)/2}\big\langle \eta(T)-\int^{|x|}_{0}K(\tau)d\tau\big\rangle^{(n-3)/2-q}\ .
\eneq
\end{enumerate}
\end{lem}
\begin{proof}
To show \eqref{5.4-2},  we have 
$$e^{-\lambda(\eta(t)+R_{1})}
 \phi_{\lambda}(x)
\ge c>0\ , \forall \la \in [\frac{\lambda_{0}}{2\langle t\rangle},
 \frac{\lambda_{0}}{\langle t\rangle}]\ ,
 $$
 by the lower bound of
\eqref{1.50}.
Then
\begin{eqnarray*}
\xi_{q}( x, T, t) 
  & = & \int_0^{\la_0} 
\frac{
1-e^{-2\lambda(\eta(T)-\eta(t))}
}{2(T-t)} e^{-\lambda(\eta(t)+R_{1})}
 \phi_{\lambda} \lambda^{q-1}
 d\la\\
 &\ge &
c \int^{\frac{\lambda_{0}}{\langle t\rangle}}_{\frac{\lambda_{0}}{2\langle t\rangle}}
\frac{
1-e^{-2\lambda(\eta(T)-\eta(t))}
}{2(T-t)} \lambda^{q-1}
 d\la\\
  &\ge &
c \int^{\frac{\lambda_{0}}{\langle t\rangle}}_{\frac{\lambda_{0}}{2\langle t\rangle}}
\frac{
1-e^{-\lambda_0(\eta(T)-\eta(t))/\<t\>}
}{2(T-t)} \lambda^{q-1}
 d\la\\
  &\gtrsim &
\frac{
1-e^{-\lambda_0\de_1 ( T - t )/\<t\>}
}{2(T-t)} \<t\>^{-q}\gtrsim
\<T\>^{-1}\<t\>^{-q}\ ,
\end{eqnarray*}
where we have used 
the elementary inequality
$
1-e^{-\lambda_0\de_1 ( T - t )/\<t\>
}  \gtrsim
(T-t)/\<T\>$ for any $t\in [0, T]$
and
the fact  that $\eta(T)-\eta(t)\ge  \delta_{1} (T-t)$.

For \eqref{5.4-3}, we 
divide the region $\int^{|x|}_{0}K(\tau)d\tau\leq \eta(T)+R_{1}$
into two parts. \\
Case A: $\int^{|x|}_{0}K(\tau)d\tau\leq \frac{\eta(T)+R_{1}}{2}$, then
$\<\eta(T)-\int^{|x|}_{0}K(\tau)d\tau\>\simeq \<T\>$.
By
 the upper bound of
 Lemma \ref{elp}, we have 
\beeq
\nonumber
\xi_{q}(x, T, T)\les\int^{\lambda_{0}}_{0}e^{-\frac{\lambda}{2}(\eta(T)+R_{1})}\lambda^{q}d\lambda
\les\int^{\lambda_{0}}_{0}e^{-\frac{\lambda}{2}(\delta_{1}T+R_{1})}\lambda^{q}d\lambda
\les \langle T\rangle^{-q-1}\ .
\eneq
Case B: $\int^{|x|}_{0}K(\tau)d\tau\geq \frac{\eta(T)+R_{1}}{2}$, then
$|x|\sim \<T\>$.
By Lemma \ref{elp}, we have 
\begin{align*}
\xi_{q}(x, T, T)\les& \int^{\lambda_{0}}_{0}e^{-\lambda\big(\eta(T)+R_{1}-\int^{|x|}_{0}K(\tau)d\tau\big)}
(\la |x|)^{-\frac{n-1}{2}}\lambda^{q}
d\lambda\\
\les &\langle x\rangle^{-\frac{n-1}{2}}\left\<\eta(T)-\int^{|x|}_{0}K(\tau)d\tau\right\>^{-q+\frac{n-3}{2}}\ .
\end{align*}
In summary, we get the desired upper bound \eqref{5.4-3} of $\xi_{q}(x, T, T)$.
\end{proof}

\subsubsection{Proof of \eqref{eq-life-afterchange}}

Let $p=p_c(n)>1+\frac{2}{n-1}$ and
$q=\frac{n-1}2-\frac 1p>\max(0, \frac{n-3}2)$, so that
$$(n-1)(1-\frac p2)-q=
\frac{n-1}2+\frac 1p-p\frac{n-1}2
=\frac 1p-\frac{n-1}2(p-1)
=\frac 1p-\frac{p+1}{p}
=
-1\ .$$
Recall from \eqref{jian1}, 
\eqref{eq-lb2} and
\eqref{app10}, we have
$$\int_{\R^{n}}|u(t)|^{p}d v_{\gm}\gtrsim
\ep^p\<t\>^{(n-1)(1-\frac p2)} ,\ \forall t
\in [1, S_{\ep}) .
$$
 By Lemma \ref{est-xi} 1) and
\eqref{eq-F-int}, we get
$$
F(T)\gtrsim \int^{T}_{1}(T-t) \ep^p \<T\>^{-1}\<t\>^{(n-1)(1-\frac p2)-q} dt 
=\int^{T}_{1}(T-t) \ep^p \<T\>^{-1}\<t\>^{-1} dt 
\ , $$
 for $T\ge 1$.
Then, for any $\theta\in (0,1)$,
we have
$$\int^{T}_{1}(T-t) \<t\>^{-1} dt 
\ges 
\int^{T}_{1} \ln t dt 
\ge 
\int^{T}_{\theta T} \ln t dt 
\ge 
(1-\theta)T \ln (\theta T)\ ,
$$
and
$F(T)\gtrsim_\theta \ep^p \ln (\theta T)$.
 In particular,
when $T\ge 1+1/2=3/2$, we have
\beeq\label{eq-slicing-1} F(T)\gtrsim \ep^p \ln (2T/3)\ .\eneq

To connect the right of \eqref{eq-F-int} with $F$, we use H\"older's inequality and  Lemma \ref{est-xi} 1), 2)  to obtain
\beeq\label{eq-slicing-2} F(T)\gtrsim \frac{1}{\langle T \rangle}\int^{T}_{0}\frac{T-t}{\langle t\rangle}\frac{F(t)^{p}}{( \log\langle t\rangle)^{p-1}}dt\ , T\in (0, S_{\ep})\ .\eneq

Based on
\eqref{eq-slicing-1}
and
\eqref{eq-slicing-2},
we could run the
``slicing method" of iteration in
  Agemi-Kurokawa-Takamura \cite{MR1785116}
(see \cite[Proposition 5.3]{WaYo18-1pub})
to improve the lower bounds of $F$, which ultimately
proves that, for some $B>0$,
\beeq\label{eq-slicing-3}F(T)\ge B \frac{\ln (3+T)}{[\ln (T/2)]^{1/(p-1)}}\exp \left[\frac{p^{j}}{p-1} \ln (B \ep^{p(p-1)}\ln T)\right] \ ,\eneq
for any $j\ge 1$ and $T\ge 4$.

Based on the above inequality \eqref{eq-slicing-3},
it is easy to prove
 the desired upper bound of lifespan
 \eqref{eq-life-afterchange} with $C\ge 2 B^{-1}$. Assume by contradiction that $$S_\ep>\exp(2 B^{-1}\ep^{-p(p-1)})\ ,$$
then we could choose $T=\exp(2 B^{-1}\ep^{-p(p-1)}) \ge 4$ for 
$0<\ep<\ep_1\ll 1$. By \eqref{eq-slicing-3}, for this choice of $T$, we have
\beeq\label{eq-slicing-4}F(T)\ge B \frac{\ln (3+T)}{[\ln (T/2)]^{1/(p-1)}}\exp \left[\frac{p^{j}}{p-1} \ln 2\right], \forall j\ge 1\ ,\eneq
and so, by letting $j\to \infty$,
$F(T)=\infty$.
This gives the desired contradiction, which completes the proof
 of 
 \eqref{eq-life-afterchange}.


\begin{thebibliography}{10}

\bibitem{MR1785116}
Rentaro Agemi, Yuki Kurokawa, and Hiroyuki Takamura.
\newblock Critical curve for {$p$}-{$q$} systems of nonlinear wave equations in
  three space dimensions.
\newblock {\em J. Differential Equations}, 167(1):87--133, 2000.

\bibitem{CaGe06}
Davide Catania and Vladimir Georgiev.
\newblock Blow-up for the semilinear wave equation in the {S}chwarzschild
  metric.
\newblock {\em Differential Integral Equations}, 19(7):799--830, 2006.

\bibitem{ChristK01}
Michael Christ and Alexander Kiselev.
\newblock Maximal functions associated to filtrations.
\newblock {\em J. Funct. Anal.}, 179(2):409--425, 2001.

\bibitem{CoLe55} 
Earl A. Coddington and
 Norman Levinson.
 \newblock {\em  Theory of ordinary differential equations}. McGraw-Hill Book Company, Inc., New York-Toronto-London, 1955. 



\bibitem{Evans10}
Lawrence~C. Evans.
\newblock {\em Partial differential equations}, volume~19 of {\em Graduate
  Studies in Mathematics}.
\newblock American Mathematical Society, Providence, RI, second edition, 2010.

\bibitem{GLS97}
Vladimir Georgiev, Hans Lindblad, and Christopher~D. Sogge.
\newblock Weighted {S}trichartz estimates and global existence for semilinear
  wave equations.
\newblock {\em Amer. J. Math.}, 119(6):1291--1319, 1997.

\bibitem{MR1814364}
David Gilbarg and Neil~S. Trudinger.
\newblock {\em Elliptic partial differential equations of second order}.
\newblock Classics in Mathematics. Springer-Verlag, Berlin, 2001.
\newblock Reprint of the 1998 edition.



\bibitem{IkSo18}
Masahiro Ikeda and Motohiro Sobajima.
\newblock Life-span of solutions to semilinear wave equation with
  time-dependent critical damping for specially localized initial data.
\newblock {\em Math. Ann.}, 372(3-4):1017--1040, 2018.

\bibitem{MaTuWa}
Masahiro Ikeda, Ziheng Tu, and Kyouhei Wakasa.
\newblock Small data blow-up of semilinear wave equation with scattering
  dissipation and time-dependent mass.
\newblock Preprint. ArXiv 1904.09574, 2019.

\bibitem{IKTW}
Takuto Imai, Masakazu Kato, Hiroyuki Takamura and Kyouhei Wakasa.
\newblock
The sharp lower bound of the lifespan of solutions to semilinear wave equations with low powers in two space dimensions.
\newblock  {\em Asymptotic Analysis for Nonlinear Dispersive and Wave Equations}, 31--53, Mathematical Society of Japan, Tokyo, Japan, 2019.




\bibitem{John79}
Fritz John.
\newblock Blow-up of solutions of nonlinear wave equations in three space
  dimensions.
\newblock {\em Manuscripta Math.}, 28(1-3):235--268, 1979.

\bibitem{LaiTa18}
Ning-An Lai and Hiroyuki Takamura.
\newblock Blow-up for semilinear damped wave equations with subcritical
  exponent in the scattering case.
\newblock {\em Nonlinear Anal.}, 168:222--237, 2018.

\bibitem{LaiZhou14}
Ning-An Lai and Yi~Zhou.
\newblock An elementary proof of {S}trauss conjecture.
\newblock {\em J. Funct. Anal.}, 267(5):1364--1381, 2014.

\bibitem{LaiZhou18}
Ning-An Lai and Yi~Zhou.
\newblock Blow up for initial boundary value problem of critical semilinear
  wave equation in two space dimensions.
\newblock {\em Commun. Pure Appl. Anal.}, 17(4):1499--1510, 2018.

\bibitem{LinLaiMing19}
Y.~Lin, Ning-An Lai, and S.~Ming.
\newblock Lifespan estimate for semilinear wave equation in schwarzschild
  spacetime.
\newblock Applied Mathematics Letters, 
 to appear. 

\bibitem{L1990}
Hans Lindblad.
\newblock Blow-up for solutions of $\Box u=|u|^{p}$ with small initial data.
\newblock{ \em Comm. Partial Differential Equations}, 15(6):757--821, 1990.


\bibitem{LMSTW}
Hans Lindblad, Jason Metcalfe, Christopher~D. Sogge, Mihai Tohaneanu, and
  Chengbo Wang.
\newblock The {S}trauss conjecture on {K}err black hole backgrounds.
\newblock {\em Math. Ann.}, 359(3-4):637--661, 2014.

\bibitem{LdSo96}
Hans Lindblad and Christopher~D. Sogge.
\newblock Long-time existence for small amplitude semilinear wave equations.
\newblock {\em Amer. J. Math.}, 118(5):1047--1135, 1996.

\bibitem{LW2018}
Mengyun Liu and Chengbo Wang.
\newblock Global existence for semilinear damped wave equations in relation
  with the strauss conjecture.
\newblock 
 {\em
 Discrete Contin. Dyn. Syst.} 40 (2020), no. 2, 709--724.

\bibitem{MW17}
Jason Metcalfe and Chengbo Wang.
\newblock The {S}trauss conjecture on asymptotically flat space-times.
\newblock {\em SIAM J. Math. Anal.}, 49(6):4579--4594, 2017.

\bibitem{Scha85}
Jack Schaeffer.
\newblock The equation {$u_{tt}-\Delta u=|u|^p$} for the critical value of
  {$p$}.
\newblock {\em Proc. Roy. Soc. Edinburgh Sect. A}, 101(1-2):31--44, 1985.

\bibitem{Sideris84}
Thomas~C. Sideris.
\newblock Nonexistence of global solutions to semilinear wave equations in high
  dimensions.
\newblock {\em J. Differential Equations}, 52(3):378--406, 1984.

\bibitem{SSW12}
Hart~F. Smith, Christopher~D. Sogge, and Chengbo Wang.
\newblock Strichartz estimates for {D}irichlet-wave equations in two dimensions
  with applications.
\newblock {\em Trans. Amer. Math. Soc.}, 364(6):3329--3347, 2012.

\bibitem{MoWa18}
Motohiro Sobajima and Kyouhei Wakasa.
\newblock Finite time blowup of solutions to semilinear wave equation in an
  exterior domain.
\newblock Preprint. ArXiv 1812.09182, 2018.

\bibitem{SW10}
Christopher~D. Sogge and Chengbo Wang.
\newblock Concerning the wave equation on asymptotically {E}uclidean manifolds.
\newblock {\em J. Anal. Math.}, 112:1--32, 2010.

\bibitem{Strauss81}
Walter~A. Strauss.
\newblock Nonlinear scattering theory at low energy.
\newblock {\em J. Funct. Anal.}, 41(1):110--133, 1981.

\bibitem{Taka15}
Hiroyuki Takamura.
\newblock Improved {K}ato's lemma on ordinary differential inequality and its
  application to semilinear wave equations.
\newblock {\em Nonlinear Anal.}, 125:227--240, 2015.

\bibitem{TakaWakasa11}
Hiroyuki Takamura and Kyouhei Wakasa.
\newblock The sharp upper bound of the lifespan of solutions to critical
  semilinear wave equations in high dimensions.
\newblock {\em J. Differential Equations}, 251(4-5):1157--1171, 2011.

\bibitem{Ta02}
Daniel Tataru.
\newblock Strichartz estimates for second order hyperbolic operators with
  nonsmooth coefficients. {III}.
\newblock {\em J. Amer. Math. Soc.}, 15(2):419--442, 2002.

\bibitem{TuLin17p1}
Ziheng Tu and Jiayun Lin.
\newblock A note on the blowup of scale invariant damping wave equation with
  sub-strauss exponent.
\newblock Preprint ArXiv:1709.00866, 2017.

\bibitem{WaYo18-1pub}
Kyouhei Wakasa and Borislav Yordanov.
\newblock Blow-up of solutions to critical semilinear wave equations with
  variable coefficients.
\newblock {\em J. Differential Equations}, 266(9):5360--5376, 2019.

\bibitem{WaYo18-2-pub}
Kyouhei Wakasa and Borislav Yordanov.
\newblock On the nonexistence of global solutions for critical semilinear wave
  equations with damping in the scattering case.
\newblock {\em Nonlinear Anal.}, 180:67--74, 2019.

\bibitem{W17}
Chengbo Wang.
\newblock Long-time existence for semilinear wave equations on asymptotically
  flat space-times.
\newblock {\em Comm. Partial Differential Equations}, 42(7):1150--1174, 2017.

\bibitem{Wang18}
Chengbo Wang.
\newblock Recent progress on the strauss conjecture and related problems.
\newblock {\em SCIENTIA SINICA Mathematica}, 48(1):111--130, 2018.

\bibitem{WaYu11}
Chengbo Wang and Xin Yu.
\newblock Concerning the {S}trauss conjecture on asymptotically {E}uclidean
  manifolds.
\newblock {\em J. Math. Anal. Appl.}, 379(2):549--566, 2011.

\bibitem{YorZh06}
Borislav~T. Yordanov and Qi~S. Zhang.
\newblock Finite time blow up for critical wave equations in high dimensions.
\newblock {\em J. Funct. Anal.}, 231(2):361--374, 2006.

\bibitem{MR1177534}
Yi~Zhou.
\newblock Blow up of classical solutions to {$\square u=|u|^{1+\alpha}$} in
  three space dimensions.
\newblock {\em J. Partial Differential Equations}, 5(3):21--32, 1992.

\bibitem{MR1233659}
Yi~Zhou.
\newblock Life span of classical solutions to {$\square u=|u|^p$} in two space
  dimensions.
\newblock {\em Chinese Ann. Math. Ser. B}, 14(2):225--236, 1993.
\newblock A Chinese summary appears in Chinese Ann. Math. Ser. A {{\bf{1}}4}
  (1993), no. 3, 391--392.

\bibitem{Zh07}
Yi~Zhou.
\newblock Blow up of solutions to semilinear wave equations with critical
  exponent in high dimensions.
\newblock {\em Chin. Ann. Math. Ser. B}, 28(2):205--212, 2007.

\bibitem{ZhouHan11}
Yi~Zhou and Wei Han.
\newblock Blow-up of solutions to semilinear wave equations with variable
  coefficients and boundary.
\newblock {\em J. Math. Anal. Appl.}, 374(2):585--601, 2011.

\bibitem{ZhouHan14criti}
Yi~Zhou and Wei Han.
\newblock Life-span of solutions to critical semilinear wave equations.
\newblock {\em Comm. Partial Differential Equations}, 39(3):439--451, 2014.

\end{thebibliography}

\end{document}